\documentclass[10pt]{amsart} 
 \normalfont

\usepackage[psamsfonts]{amssymb}
\usepackage{pifont}
\theoremstyle{plain}

\newtheorem{theorem}{Theorem}[section]
\newtheorem{lemma}[theorem]{Lemma}
\newtheorem{proposition}[theorem]{Proposition}

\newtheorem{corollary}[theorem]{Corollary}

\def\bc{\mathbb{C}}

\def\br{\mathbb{R}}

\def\fG{\mathfrak{G}}
\def\fH{\mathfrak{H}}
\def\fK{\mathfrak{K}}

\def\fM{\mathfrak{M}}

\def\fP{\mathfrak{P}}

\def\fZ{\mathfrak{Z}}
\def\fs{\mathfrak{s}}
\def\fo{\mathfrak{o}}

\newcounter{commentlabel}
\begin{document}
  \title{Growth estimates for orbits of self adjoint groups}
\author{Patrick Eberlein}
\date{\today}
\address{Department of Mathematics, University of North Carolina, Chapel Hill, NC 27599}
\email{pbe@email.unc.edu}
\thanks{AMS Subject Classifications : 53C99 ; 37A35 ; 22E99 ; 22F99 ; 20G99\\
Keywords : Self adjoint group actions ; closed orbits , asymptotic behavior of g $\rightarrow |g(v)|$}
\maketitle

$\mathit{Abstract}$

	Let G denote a closed, connected, self adjoint, noncompact subgroup of $GL(n,\br)$, and let $d_{R}$ and $d_{L}$ denote respectively the right and left invariant Riemannian metrics defined by the canonical inner product on $M(n,\br) = T_{I} GL(n,\br)$.  Let v be a nonzero vector of $\br^{n}$ such that the orbit G(v) is unbounded in $\br^{n}$.  Then the function $g \rightarrow d_{R}(g, G_{v})$ is unbounded, where $G_{v} = \{g(v) = v \}$, and we obtain algebraically defined upper and lower bounds $\lambda^{+}(v)$ and $\lambda^{-}(v)$ for the asymptotic behavior of the function $\frac{log|g(v)|}{d_{R}(g, G_{v})}$ as $d_{R}(g, G_{v}) \rightarrow \infty$.  The upper bound $\lambda^{+}(v)$ is at most 1.  The orbit G(v) is closed in $\br^{n} \Leftrightarrow \lambda^{-}(w)$ is positive for some w $\in$ G(v).  If $G_{v}$ is compact, then $g \rightarrow |d_{R}(g,I) - d_{L}(g,I)|$ is uniformly bounded in G, and the exponents $\lambda^{+}(v)$ and $\lambda^{-}(v)$ are sharp upper and lower asymptotic bounds for the functions $\frac{log|g(v)|}{d_{R}(g,I)}$ and  $\frac{log|g(v)|}{d_{L}(g,I)}$ as $d_{R}(g,I) \rightarrow \infty$ or as $d_{L}(g,I) \rightarrow \infty$.  However, we show by example that if $G_{v}$ is noncompact, then there need not exist asymptotic upper and lower bounds for the function  $\frac{log|g(v)|}{d_{L}(g, G_{v})}$ as $d_{L}(g, G_{v}) \rightarrow \infty$.  The results apply to  representations of noncompact semisimple Lie groups G on finite dimensional real vector spaces.

\section{Basic objects and notation}

$\mathit{Self~adjoint~subgroups~of~ GL(n,\br)}$

	Let $M(n,\br)$ denote the n x n real matrices, and let $GL(n,\br)$ denote the group of invertible matrices in $M(n,\br)$. Let G denote a closed, connected subgroup of $GL(n,\br)$.  The Lie algebra $\fG$ of G in $M(n,\br)$ is given by $\fG = \{X \in M(n,\br) : exp(tX) \in G~\rm{for~all~t \in \br}~ \}$, where exp : $M(n,\br) \rightarrow GL(n,\br)$ denotes the matrix exponential map. It is known that every closed subgroup of $GL(n,\br)$ is a Lie group  with the subspace topology.  Let $O(n, \br) = \{g \in GL(n, \br) : gg^{t} = g^{t}g = I \}$, where I is the identity matrix, and let $\fs \fo(n, \br) =\{X \in M(n, \br) : X^{t} = - X \}$, the Lie algebra of $O(n, \br)$.
	
	A subgroup G of $GL(n,\br)$ is said to be $\mathit{self~adjoint}$ if $g^{t}\in$ G whenever $g \in$ G. .  In this paper G will typically denote a connected, closed, self adjoint, noncompact subgroup of $GL(n,\br)$ unless stated otherwise.  We define $K = G~ \cap~ O(n,\br)$.  The corresponding Lie algebra is $\fK = \fG~ \cap~ \fs \fo (n,\br)$.  Let $\fP = \{X \in \fG : X^{t} = X \}$.  If G is self adjoint, then $\fG = \fK \oplus \fP$.
\newline

$\mathit{Remark}$  Let G be a connected, noncompact, semisimple Lie group, and let V be a finite dimensional real vector space.  Let $\rho : G \rightarrow GL(V)$ be a $C^{\infty}$ homomorphism.  Then there exists an inner product $\langle ,\rangle$ on V such that $\rho(G)$ is self adjoint on V relative to $\langle , \rangle$.  Hence the results of this paper can be applied to $\rho(G)$.  See (10.3) for details.
\vspace{.3in}

$\mathit{Left~and~right~invariant~Riemannian~metrics}$
	
	Let $\langle , \rangle$ denote the canonical inner product on $M(n,\br) \approx T_{I}GL(n,\br)$ given by $\langle A , B \rangle = trace~AB^{t}$.  For a closed, connected subgroup G of $GL(n,\br)$ the inner product $\langle , \rangle$ defines an inner product on $T_{I}G$ and a right invariant Riemannian structure, also denoted $\langle , \rangle$, on G.  Let $d_{R}$ denote the corresponding right invariant Riemannian metric on G.  Similarly, the canonical inner product $\langle , \rangle$ on $T_{I}G$ defines a left invariant Riemannian metric $d_{L}$ on G.  We call $d_{R}$ and $d_{L}$ the $\mathit{canonical}$ right and left invariant Riemannian metrics on G.
\newline

$\mathit{Stabilizer~subgroups~and~associated~objects}$
	
	For each nonzero vector v of $\br^{n}$ let $G_{v} = \{g \in G : g(v) = v \}$, and let $\fG_{v} = \{X \in \fG : X(v) = 0 \}$ denote the Lie algebra of $G_{v}$.  Let $\fK_{v} = \fs \fo (n,\br) \cap \fG_{v} = \{X \in \fG_{v} : X^{t} = - X \}$, and let $\fP_{v} = \{X \in \fG_{v} : X^{t} = X \}$.  Let $\widetilde{\fP_{v}} = (\fK + \fG_{v})^{\perp} = \{\zeta \in \fG : \langle \zeta , \eta \rangle = 0~\rm{for~all}~\eta \in \fK + \fG_{v} \}$.  Note that $\widetilde{\fP_{v}}$ is orthogonal to $\fP_{v}$, and  $\widetilde{\fP_{v}} \subset \fP$ since $\langle \fK , \fP \rangle = 0$ (cf. (4.2)).
\newline

$\mathit{The~ growth~exponents~\lambda^{-}(v)~,~\lambda^{+}(v)}~and~ minimal~ vectors$
	  	
	   For each X $\in \fP, \br^{n}$ is an orthogonal direct sum of the eigenspaces of X.  For a nonzero element v in $\br^{n}$ and a nonzero element X of $\fP$ we define $\lambda_{X}(v)$ to be the largest eigenvalue $\lambda$ of X for which v has a nonzero component in $V_{\lambda} = \{w \in \br^{n} : X(w) = \lambda w \}$.
	
	For a nonzero vector v of $\br^{n}$ we define
	
	$\lambda^{-}(v) = inf~\{\lambda_{X}(v) : X \in \widetilde{\fP_{v}}~\rm{and}~|X| = 1 \}$
	
	  $\lambda^{+}(v) = sup~\{\lambda_{X}(v) : X \in \widetilde{\fP_{v}}~\rm{and}~|X| = 1 \}$
\newline

	  A vector v in $\br^{n}$ is said to be $\mathit{minimal}$ (for the action of G) if $|g(v)| \geq |v|$ for all g $\in$ G.  If v $\in \br^{n}$ is minimal, then $\fG_{v} = \fK_{v} \oplus \fP_{v}$, orthogonal direct sum (section 6). It is known (cf. Theorem 4.4 of [RS]) that G(v) is closed in $\br^{n}$ if v is minimal, and conversely, if G(v) is closed in $\br^{n}$, then it is easy to show that G(v) contains a minimal vector w.
 \newline

$\mathit{Remark}$

 If X $\in \fP_{v}$, then it follows from the definitions that $X(v) = 0$ and  $\lambda_{X}(v) = 0$.  By definition $\widetilde{\fP_{v}}$ is orthogonal to $\fP_{v}$. Hence, in the definition of $\lambda^{-}(v)$, when we restrict consideration to unit vectors X in $ \widetilde{\fP_{v}}$  we allow the possibility that $\lambda^{-}(v) > 0$.  If this happens then the orbit G(v) is closed in $\br^{n}$ as we shall see in (10.2).
	  	
\section{The main result and its consequences}

	Define $\underline{lim}~_{d_{R}(g, G_{v}) \rightarrow \infty}~\frac{log|g(v)|}{d_{R}(g, G_{v})}$ (respectively $\overline{lim}~_{d_{R}(g, G_{v}) \rightarrow \infty}~\frac{log|g(v)|}{d_{R}(g, G_{v})}$) to be the smallest (respectively largest) limit of a sequence $\frac{log |g_{k}(v)|}{d_{R}(g_{k}, G_{v})}$, where $\{ g_{k} \}$ is any sequence in G such that $d_{R}(g_{k}, G_{v}) \rightarrow \infty$.
\newline
	
	  The main result of this paper is the following

\begin{theorem}  Let G be a closed, connected, self adjoint, noncompact subgroup of $GL(n,\br)$.  Let $d_{R}$ denote the canonical right invariant Riemannian metric on G.  Let v be a nonzero  vector in $\br^{n}$ such that the orbit G(v) is unbounded.  Then

	1)  The function $g \rightarrow d_{R}(g, G_{v})$ is unbounded on G.
	
	2)  If v is minimal, then $\lambda^{-}(v) > 0$.  For arbitrary nonzero v we have $-1 \leq \lambda^{-}(v) \leq \lambda^{+}(v) \leq 1$.
	
	3)  $\lambda^{-}(v) \leq \underline{lim} \hspace{.05in} _{d_{R}(g, G_{v}) \rightarrow \infty} \hspace{.1in} \frac{log|g(v)|}{d_{R}(g, G_{v})} \leq  \overline{lim} \hspace{.05in}  _{d_{R}(g, G_{v}) \rightarrow \infty} \hspace{.1in}\frac{log|g(v)|}{d_{R}(g, G_{v})} \leq \lambda^{+}(v)$
\end{theorem}

$\mathit{Remarks}$

	1)  Let c be a positive constant such that $d_{R}(k,I) \leq c$ for all k $\in$ K, where I denotes the identity matrix.  The triangle inequality implies that $|d_{R}(g,G_{v}) - d_{R}(g, K \cdot G_{v}| \leq c$ for all g $\in$ G.  Hence we could replace $d_{R}(g, G_{v})$ by $d_{R}(g, K \cdot G_{v}$ in the statement of the result above.  This replacement is convenient for the proof of (2.1).  In section 5 we show that if v is a nonzero minimal vector, then $K \cdot G_{v}$ is the minimum set and also the set of critical points for the function $F_{v} : G \rightarrow \br$ given by $F_{v}(g) = |g(v)|^{2}$.

	2)  A sharper version of the inequality for $\lambda^{-}(v)$ in assertion 3) is obtained in (8.2), but the statement of that result involves a quantity that is more difficult to define and compute than $\lambda^{-}(v)$.  It may be the case that the bounds $\lambda^{-}(v)$ and $\lambda^{+}(v)$ are sharp in the main result, but we are only able to prove this in the case that $G_{v}$ is compact.
	
	3)  The set of vectors v in V for which $dim~ G_{v} \leq dim~ G_{v'}$ for all $v' \in$ V is a nonempty G-invariant Zariski open subset O.  The function $\lambda^{-}$ is lower semicontinuous on O (cf. (9.1)) and by 2) of (2.1) $\lambda^{-}$ is positive on the set of minimal vectors in V whose G-orbits in V are unbounded.  By (7.3)  $\lambda^{-}$ is K-invariant, but it is unclear if $\lambda^{-}$ is G-invariant.

    4)  If $\fM$ denotes the set of minimal vectors in $\br^{n}$, then for a compact subset C of $\fM~\cap~O$ and v $\in$ C the functions $g \rightarrow \frac{log|g(v)|}{d_{R}(g, G_{v})}$ have a uniform lower bound.  See (9.4) for a precise statement.
\newline

$\mathit{Closed~orbits}$

    5)  If G(v) is closed in $\br^{n}$ for some nonzero vector v, then G(v) contains a minimal vector w by (6.3) and $\lambda^{-}(w) > 0$ by (7.6).  Conversely, if $\lambda^{-}(w) > 0$ for some nonzero vector w, then G(w) is closed in $\br^{n}$ by (2.2) and (2.5).  Hence G has a nontrivial closed orbit in $\br^{n} \Leftrightarrow \lambda^{-}(w) > 0$ for some nonzero vector w of $\br^{n}$.  If G does not have a nontrivial closed orbit in $\br^{n}$, then the zero vector lies in $\overline{G(v)}$ for every v $\in \br^{n}$ by Lemma 3.3 of [RS].

    6)  Let O be the nonempty G-invariant Zariski open subset of $\br^{n}$ that is defined above in 3).  Under either of the following hypotheses  there exists a nonempty G-invariant Zariski open subset U of $\br^{n}$ such that $U \subset O$ and G(w) is closed in $\br^{n}$ for every w $\in$ U.   The second hypothesis appears in the statement of (2.3) below.

\hspace{.2in} 	a)  $\lambda^{-}(v) > 0$ for some v $\in$ O.
	
\hspace{.2in}	b)  $G_{v}$ is compact for some nonzero vector v $\in \br^{n}$.

    In case a) the orbit G(v) is closed in $\br^{n}$ by (2.2) and (2.5), and it has maximal dimension among the G-orbits by the definition of O.  The assertion of 6) in case a) is now well known in the complex case where $G \subset GL(n,\bc)$ acts on $\bc^{n}$.  For a proof in the real case see, for example, Proposition 2.1 of [EJ].  The assertion of 6) in case b) is proved in Proposition 2.6 of [EJ].
\newline

	The next result, which we prove in (6.4), proves the first statement of (2.1) and also shows that an orbit G(v) is bounded in $\br^{n} \Leftrightarrow$ G(v) is compact.
	
	\begin{proposition}  Let G denote a closed, connected, noncompact subgroup of $GL(n,\br)$, and let v be a nonzero vector of $\br^{n}$.  Then the following are equivalent.

	1)  The orbit G(v) is bounded in $\br^{n}$.
	
	2)  The function $g \rightarrow d_{R}(g, G_{v})$ is bounded on G.
	
	3)  $G = K \cdot G_{v}$
	
	4)  $X(v) = 0$ for any X $\in \fP$.
	
\noindent If G has no nontrivial compact, connected, normal subgroups, then G fixes v if any of the conditions above hold.
\end{proposition}
	
	If $G_{v}$ is compact, then we may sharpen (2.1) to obtain something valid for both $d_{R}$ and $d_{L}$.
	
\begin{corollary}  Let G be a closed, connected, self adjoint, noncompact subgroup of $GL(n,\br)$.  Let $d_{R}$ and $d_{L}$ denote respectively the canonical right invariant and left invariant Riemannian metrics on G.  Let I denote the identity matrix in $GL(n,\br)$.  Let v be a nonzero  vector such that the orbit G(v) is unbounded and $G_{v}$ is compact.  Then

	   $\lambda^{-}(v) = \underline{lim} \hspace{.05in} _{d_{R}(g, I) \rightarrow \infty} \hspace{.1in} \frac{log|g(v)|}{d_{R}(g, I)} = \underline{lim} \hspace{.05in} _{d_{L}(g, I) \rightarrow \infty} \hspace{.1in} \frac{log|g(v)|}{d_{L}(g, I)}\leq  \overline{lim} \hspace{.05in}  _{d_{R}(g, I) \rightarrow \infty} \hspace{.1in}\frac{log|g(v)|}{d_{R}(g, I)} = \overline{lim} \hspace{.05in}  _{d_{L}(g, I) \rightarrow \infty} \hspace{.1in}\frac{log|g(v)|}{d_{L}(g, I)} = \lambda^{+}(v)$
\end{corollary}

$\mathit{Remarks}$
	
	1)  The replacement of $d_{R}(g, G_{v})$ by $d_{L}(g, G_{v})$ in the statement of (2.1) does not always hold if $G_{v}$ is noncompact.  See Appendix I for an example.  Note that G acts on $\br^{n}$ on the left, and this may be relevant to the asymmetry of $d_{R}$ and $d_{L}$ in (2.1).
	
	2)  Even in the case that $G_{v}$ is compact the quantities in the main result involving $d_{R}(g,I)$ cannot be replaced conveniently by estimates involving $|g|$, where $|A|^{2} = trace~ AA^{t}$ for A $\in M(n,\br)$.  This may reflect the fact that $g \rightarrow d_{R}(g,I)$ and $g \rightarrow |g|$ are the distances between g and the identity in the typically nonabelian group $(G,d_{R})$ and the abelian group $M(n,\br)$ with the Euclidean distance.  See the end of section 5 for details.
\newline

	The main tools needed for the proof of the main result are the KP decomposition of G, (5.4), and the following generalization of it that is proved in (5.6).
	
\begin{lemma}  Let G,v and $d_{R}$ be as in (2.1).  For every g $\in$ G there exist elements k $\in$ K, h $\in G_{v}$ and $X \in \widetilde{\fP_{v}}$ such that $g = k~ exp(X)~ h$ and $|X| = d_{R}   (g,K \cdot G_{v}) = d_{R}(exp(X), K \cdot G_{v}) = d_{R}(exp(X), Id)$.
\end{lemma}

	  From (2.1) we obtain the following consequence (cf. (10.2)).	
 		
\begin{proposition}  Let G be a closed, self adjoint, noncompact subgroup of $GL(n,\br)$ with finitely many connected components, and let $G_{0}$ denote the connected component that contains the identity.   Let v be a nonzero vector of $\br^{n}$ such that the orbit G(v) is unbounded.  Then the following assertions are equivalent.

 	1)  G(v) is closed in $\br^{n}$.
	
	2)  $\lambda^{-}(w) > 0$ for some w $\in$ $G_{0}$(v), where $\lambda^{-} : \br^{n} \rightarrow \br$ is defined by the connected group $G_{0}$.
	
	3)  G$_{0}$(v) is closed in $\br^{n}$.
	
	4)  G(v) contains an element w that is minimal with respect to both G$_{0}$ and G.
\end{proposition}

$\mathit{Remarks}$

	1)  If G(v) is bounded, then G$_{0}$(v) is also bounded.  By 3) of (2.2) $G_{0}(v) = K(v)$, where $K = G_{0} \cap O(n,\br)$.  In particular every element of $G_{0}(v)$ is minimal with respect to $G_{0}$.§
	
	2)  The equivalence of 1) and 4) in (2.5) is Theorem 4.4 of [RS]), but the proof here is different and more elementary.  Although we are guided by the logical development of [RS], especially in section 6, the arguments used here are differential geometric in nature.  In addition they require some knowledge of Lie groups but no knowledge of algebraic groups or algebraic geometry.
\newline	

$\mathit{Comparison~ between~ \lambda^{-}(v)~ and~ the~ Hilbert-Mumford~function~ M(v)}$

	In [KN]  G. Kempf and L. Ness introduced a Hilbert- Mumford function M : V $\rightarrow \bc$ defined on the orbits of a reductive complex algebraic group G acting on a complex finite dimensional vector space V.  Later, A. Marian extended some of this work to the real setting in [M].  This function was studied further in [EJ], where it was shown, for example, that M(v) $< 0 \Leftrightarrow$ the orbit G(v) is closed and G$_{v}$ is compact.
\newline
	
	Let v be a nonzero element of $\br^{n}$, and let X be a nonzero element of $\fP$.  We define $\mu_{X}(v)$ to be the $\mathit{smallest}$ of the eigenvalues $\lambda$ of X such that v has a nonzero component in the eigenspace $V_{\lambda}$ corresponding to $\lambda$.   Define M(v) $= sup \{\mu_{X}(v) : X \in \fP, |X| = 1 \}$.
	
	From the definitions we obtain immediately that  $\lambda_{X}(v) = - \mu_{-X}(v)$ for all nonzero X in $\fP$ and all nonzero v $\in \br^{n}$.
\newline

	In (7.9) we show

\begin{proposition}  Let v be a nonzero vector of $\br^{n}$.  Then $\lambda^{-}(v) \geq - M(v)$.
\end{proposition}

$\mathit{Remark}$  The function M : V $\rightarrow \br$ is constant on G-orbits, but we are only able to prove in (7.3) that the function $\lambda^{-} : V \rightarrow \br$ is constant on K-orbits.  This is consistent with the fact that M(v) is defined by unit vectors in $\fP$ while $\lambda^{-}(v)$ is defined by unit vectors in $\widetilde{\fP_{v}}$.  There is no clear relation between $\widetilde{\fP_{v}}$ and  $\widetilde{\fP}_{g(v)}$ for g $\in$ G.

	However, the function  $\lambda^{-} : V \rightarrow \br$ is better at detecting closed orbits than  M : V $\rightarrow \br$.  By (2.5) an orbit G(v) is closed $\Leftrightarrow \lambda^{-}(w) > 0$ for some w $\in$ G(v).  In particular, by (7.6) $\lambda^{-}(v) > 0$ if G(v) is unbounded and v is minimal.  By contrast there may be nonclosed orbits G(v) on which M(v) $\equiv 0$ and closed orbits G(v) on which M(v) $\equiv 0$.

\section{Left and Right invariant geometry in Lie groups}

	Later we will consider primarily the canonical right invariant geometry of a closed subgroup G of $GL(n,\br)$, but the results of this section are valid for both left and right invariant Riemannian metrics on a connected Lie group G.    We omit the proofs of these results, which are well known.  Let e denote the identity of G.
		
	Let $\fG_{R}$ (respectively $\fG_{L}$ ) denote the real vector space of right invariant (respectively left invariant) vector fields on a connected Lie group G.  The vector space $\fG_{R}$ (respectively $\fG_{L}$ ) is closed under the usual Lie bracket of vector fields on G.  Note that $dim~\fG_{R}  = dim~\fG_{L} = dim~T_{e}G$ since the map $X \rightarrow X(e)$ is a linear isomorphism of $\fG_{R}$ ( respectively $\fG_{L}$ ) onto $T_{e}G$.
\newline
		
$\mathit{Properties~ of~ left~ and~right~invariant~ Riemannian~ metrics}$

	If $\langle , \rangle$ is an inner product on $T_{e}G$, then $\langle , \rangle$ extends uniquely to an inner product $\langle , \rangle_{g}$ on $T_{g}G$ for each g $\in$ G by defining $\langle X(g), Y(g) \rangle_{g} = \langle X(e) , Y(e) \rangle$ for all right invariant vector fields X,Y on G.  This also defines an inner product $\langle , \rangle$ on the vector space $\fG_{R}$ of right invariant vector fields on G.  Let $d_{R}$ denote the Riemannian metric on G determined by $\{\langle , \rangle_{g} : g \in G \}$.  The right translations $R_{g} : h \rightarrow hg$ are isometries of $(G, d_{R})$. Similarly, the left translations $L_{g} : h \rightarrow gh$ are isometries of $(G, d_{L})$ for all g $\in$ G.
	
\begin{proposition}  Let $\langle , \rangle$ denote a left or right invariant Riemannian structure on a connected Lie group G, and let d denote the corresponding Riemannian metric.  Then (G,d) is complete as a metric space.
\end{proposition}

$\mathit{The~ Levi-Civita~ connection~ on~ \fG_{R}}$

	For a proof of the next result see for example Proposition 3.18 of [CE].

\begin{proposition}  Let $\langle , \rangle$ denote a left or right invariant Riemannian structure on a connected Lie group G, and let $\langle , \rangle$ also denote the corresponding inner product on $\fG_{R}$ or $\fG_{L}$ .  Let $\nabla$ denote the Levi-Civita connection on TG determined by $\langle , \rangle$.  Let X,Y $\in \fG_{R}$.  Then

\hspace{1in}	$\nabla_{X} Y = \frac{1}{2} \{[X,Y] - (ad~X)^{*}(Y) - (ad~Y)^{*}(X) \}$
\end{proposition}

$\mathit{Remarks}$

	1)  (ad X)$^{*}$ and (ad Y)$^{*}$ denote the metric adjoints of the linear maps ad X, ad Y : $\fG_{R} \rightarrow \fG_{R}$ that are determined by the inner product $\langle , \rangle$ on $\fG_{R}$.  Define the metric adjoint (ad X)* analogously in $\fG_{L}$.
	
	2)  The assertion of the proposition also shows that $\nabla_{X} Y \in \fG_{R}$ (respectively $\fG_{L}$) if X,Y $\in \fG_{R}$ (respectively $\fG_{L}$).
	
\section{Left and right invariant geometry in $GL(n,\br)$}

	Let exp : $M(n,\br) \rightarrow GL(n, \br)$ be the matrix exponential map given by $exp(X) = \sum_{n=0}^{\infty}~X^{n} / n!$.  Let $\langle , \rangle$ denote the canonical positive definite inner product on $M(n,\br)$ given by $\langle A , B \rangle = trace(AB^{t}) \newline = \sum_{i,j=1}^{n}~A_{ij} B_{ij}$.  Note that $M(n,\br)$ is the Lie algebra of $GL(n,\br)$ with the bracket [ , ] given by $[A,B] = AB - BA$.
	
	Let G be a closed, connected self adjoint subgroup of $GL(n,\br)$ and let d$_{L}$ and d$_{R}$ denote the canonical left invariant and right invariant Riemannian metrics on G.
\newline
\vspace{.5in}

$\mathit{Left~and~right~invariant~vector~ fields~in~GL(n,\br)}$

	Let I denote the identity in $M(n,\br)$.  For A $\in M(n,\br)$ let $A_{I} \in T_{I} GL(n,\br)$ denote the initial velocity of the curve $\alpha(t) = I + tA$.  We identify $M(n,\br)$ with $T_{I} GL(n,\br)$ by means of the linear isomorphism $A \rightarrow A_{I}$, and we let  $\langle , \rangle$ also denote the canonical inner product on $T_{I}GL(n,\br)$.
	
	For A $\in M(n,\br)$ let $A_{R}$ denote the right invariant vector field on $GL(n,\br)$ that satisfies the initial condition  $A_{R}(I) = A_{I}$.  Define the left invariant vector field $A_{L}$ in similar fashion.
		
\begin{proposition}	If T : G $\rightarrow$ G is the diffeomorphism given by T(g) $=$ g$^{t}$, then $T : (G,d_{L}) \rightarrow (G,d_{R})$ is an isometry.
\end{proposition}

\begin{proof}
	Observe that if $\eta_{I} \in T_{I}G$, then $T_{*}(\eta_{I}) = (\eta^{t})_{I}$, and it follows that $|\eta_{I}|_{L} = |T_{*}(\eta_{I})|_{R} = (trace~\eta~ \eta^{t})^{\frac{1}{2}}$.  In general if $\zeta \in T_{g}G$ for some g $\in$ G, then $\zeta = (L_{g})_{*}(\eta_{I})$ for a unique $\eta_{I} \in T_{I}G$.  Hence $|\zeta|_{L} = |\eta_{I}|_{L} = |T_{*}(\eta_{I})|_{R} = |T_{*}(\zeta)|_{R}$ since $T_{*}(\zeta) = (T \circ L_{g})_{*}(\eta_{I}) = (R_{g^{t}} \circ T)_{*}(\eta_{I}) = (R_{g^{t}})_{*} (T_{*}(\eta_{I}))$.
\end{proof}
		
	The next result follows from routine computations, which we omit.
	
\begin{proposition}  Let $\langle , \rangle$ denote the positive definite inner product on $M(n,\br)$ defined above.  Then

	1)  $\langle , \rangle$ is invariant under Ad k for all k $\in O(n,\br)$, where Ad k(X) = kXk$^{-1}$ for X $\in M(n,\br)$.
	
	2)  $\langle X , Y \rangle = 0$ if X,Y $\in M(n,\br)$ are elements with $X^{t} = - X$ and $Y^{t} = Y$.
	
	3)  Let (ad A)$^{*}$ denote the metric adjoint of ad A : $M(n,\br) \rightarrow M(n,\br)$ relative to the inner product $\langle , \rangle$.  Then (ad A)$^{*} =$ ad A$^{t}$ for all A $\in M(n,\br)$.
\end{proposition}
\vspace{.1in}
	
\begin{proposition}  For A,B $\in M(n,\br)$ let $[A,B] = AB - BA$ and let $[A_{L},B_{L}]$ and $[A_{R},B_{R}]$ denote the usual Lie brackets of vector fields in $GL(n,\br)$.  Then

	1)  $[A,B]_{R} = - [A_{R}, B_{R}]$.
	
	2)  $[A,B]_{L} =  [A_{L}, B_{L}]$.
\end{proposition}

\begin{proof}  Let $\{x_{ij} : 1 \leq i,j \leq n \}$ denote the standard coordinate functions of $M(n,\br)$, restricted now to $GL(n,\br)$.  If $A = (A_{ij}) \in M(n,\br)$, then routine computations show that $A_{R}(x_{ij}) = \sum_{k=1}^{n} A_{ik} x_{kj}$ and $A_{L}(x_{ij}) = \sum_{k=1}^{n} x_{ik} A_{kj}$.  The statements 1) and 2) now follow immediately since both sides of the equality assertions have the same values on the coordinate functions $\{x_{ij} \}$
\end{proof}

$\mathit{The~Lie~algebra~of~G~in~M(n,\br)}$

	Let $\fG = \{A \in M(n,\br) : A_{I} \in T_{I}G \}$. One calls $\fG$ the $\mathit{Lie~ algebra~ of~ G}$ in $M(n,\br)$.
		
\begin{proposition}  Let G be a closed, connected subgroup of $GL(n,\br)$, and let $\fG$ denote the Lie algebra of G in $M(n,\br)$.  Let A,B $\in \fG$.  Then
	
	1)  The vector fields $A_{R}$ and $A_{L}$ are tangent to G at every point of G.
	
	2)  $exp(tA) \in$ G for all t $\in \br$ and $t \rightarrow exp(tA)$ is the integral curve starting at I for the vector fields $A_{R}$ and $A_{L}$ restricted to G.
	
	3)  $[A,B] = AB - BA \in \fG$.
\end{proposition}

\begin{proof}  Assertions 1) and 2) follow from routine arguments that we omit.  Assertion 3) follows from 1) and (4.3).
\end{proof}

\begin{corollary}  Let G be a closed, connected, self adjoint, noncompact subgroup of $GL(n,\br)$.   Let $\fG$ denote the Lie algebra of G, and let  $d_{R}$ denote the canonical right invariant Riemannian metric on G.  Then

	1)  $\fG = \fK \oplus \fP$, orthogonal direct sum, where $\fK = \{X \in \fG : X^{t} = - X \}$ and $\fP = \{X \in \fG : X^{t} = X \}$
	
	2)  If A $\in \fG$, then $(ad~A_{R})^{*} = (ad~A^{t}_{R})$.
	
	3)  If A $\in \fG$, then $\gamma(t) = exp(tA)$ is a geodesic of $(G, d_{R}) \Leftrightarrow AA^{t} = A^{t}A$.

	4)  If K $= G~\cap~O(n,\br)$, then the left translation $L_{k}$ is an isometry of $(G, d_{R})$ for all k $\in$ K.
\end{corollary}	

$\mathit{Remark}$  The statements corresponding to 2), 3) and 4) for the canonical left invariant Riemannian metric $d_{L}$ on G are also true, and the proofs are essentially the same.

\begin{proof}  1)  It follows from (4.4) that $\fG^{t} = \fG$ since $G^{t} = G$.  If $X \in \fG$, then $X = K + P \in \fK \oplus \fP$, where $K = \frac{1}{2} (X - X^{t})$ and $P = \frac{1}{2} (X + X^{t})$.  The subspaces $\fK$ and $\fP$ are orthogonal by 2) of (4.2).

	2)  This follows immediately from 3) of (4.2), (4.3) and the fact that $\langle A_{R} , B_{R} \rangle = \langle A , B \rangle$ for all A,B $\in M(n,\br)$.
	
	3)  For A $\in \fG$ the curve $\gamma(t) = exp(tA)$ is an integral curve of the vector field $A_{R}$ on G by (4.4), and it follows that $\gamma$ is a geodesic of G $\Leftrightarrow \nabla_{A_{R}} A_{R} \equiv 0$.  It follows that $\nabla_{A_{R}} A_{R} = - (ad~A_{R})^{*}(A_{R})= - ad~A_{R}^{t}(A_{R})= [A^{t}, A]$ by (3.2), assertion 2) above and 1) of (4.3).  Assertion 3) now follows immediately.
	
	4)  For k $\in$ K, A $\in M(n,\br)$ and C $\in GL(n,\br)$ it is easy to show that $(L_{k})_{*}A_{R}(C) = \newline (k A k^{-1})_{R}(kC)$.  Hence for k $\in$ K, A,B $\in M(n,\br)$ and C $\in GL(n,\br)$ we have \newline $\langle (L_{k})_{*} A_{R}(C),(L_{k})_{*} B_{R}(C) \rangle = \langle (k A k^{-1})_{R}(kC), (k B k^{-1})_{R}(kC) \rangle = \langle k A k^{-1},k B k^{-1} \rangle = \newline \langle A , B \rangle = \langle A_{R}(C) , B_{R}(C) \rangle$.  We use 1) of (4.2) and the fact that $\langle A_{R} , B_{R} \rangle \equiv \langle A , B \rangle$ on $GL(n,\br)$.
\end{proof}

\section{Self adjoint subgroups of $GL(n,\br)$ and the KP decomposition}

	In this section let G be a closed, connected subgroup of $GL(n,\br)$.  The group G is a Lie group in the subspace topology of $GL(n,\br)$.
\newline
	
$\mathit{Structure~of~self~adjoint~subalgebras~of~ M(n,\br)}$

	A subalgebra $\fG$ of $M(n,\br)$ is $\mathit{self~ adjoint}$ if $X^{t} \in \fG$ whenever $X \in \fG$.  If G is a closed, connected  subgroup of $GL(n,\br)$, then it follows from (4.4) that G is self adjoint $\Leftrightarrow \fG$ is self adjoint.	 A subalgebra $\fG$ of $M(n,\br)$ is $\mathit{semisimple}$ if the Killing form B : $\fG \times \fG \rightarrow \br$ given by B(X,Y) $=$ trace ad X $\circ$ ad Y is nondegenerate.

\begin{proposition}  Let $\fG$ be a self adjoint Lie subalgebra of $M(n,\br)$.  Let $\fZ(\fG)$ denote the center of $\fG$, and let $\fG_{0}$ denote the orthogonal complement of $\fZ(\fG)$ in $\fG$ relative to the canonical inner product $\langle , \rangle$.  Then

	1)  $\fG = \fZ(\fG) \oplus \fG_{0}$.
	
	2)  $\fZ(\fG)$ and $\fG_{0}$ are self adjoint ideals of $\fG$.
	
	3)  $\fG_{0}$ is semisimple.
\end{proposition}

	See Appendix II for the proof.
\newline

$\mathit{The~case~of~an~irreducible~action}$

	Let G be a self adjoint subgroup of $GL(n,\br)$ and let V be a G-invariant subspace of $\br^{n}$.  If $V^{\perp} = \{w \in \br^{n} : \langle v , w \rangle = 0~\rm{for~all~v~\in~V} \}$, then $V^{\perp}$ is invariant under $G^{t} = G$.  It follows that $\br^{n}$ is an orthogonal direct sum of  irreducible G invariant subspaces, and the restriction of G to each such subspace V is a self adjoint subgroup of GL(V) relative to the restriction of the canonical inner product $\langle , \rangle$ to V.  We investigate the structure of G and $\fG$ restricted to V $\approx \br^{k}$, keeping in mind the fact that some normal subgroup of G may be zero when restricted to V.  Note that if G is connected, then G acts irreducibly on $\br^{n} \Leftrightarrow$ its Lie algebra $\fG$ acts irreducibly on $\br^{n}$.

\begin{proposition}  Let $\fG$ be a self adjoint Lie subalgebra of $M(n,\br)$ that acts irreducibly on $\br^{n}$, and let $\fZ(\fG)$ denote the center of $\fG$.  Then

	1)  $\fZ(\fG) = (\fZ(\fG) ~\cap~\fP) \oplus (\fZ(\fG) ~\cap~\fK)$
	
	2)  If $\fZ(\fG) ~\cap~\fP \neq \{0 \}$, then $\fZ(\fG) ~\cap~\fP = \br~Id$.
	
	3)  If $\fZ(\fG) ~\cap~\fK \neq \{0 \}$, then $\fZ(\fG) ~\cap~\fK = \br~J$, where $J^{2} = -~Id$.
\end{proposition}

\begin{proof}  Assertion 1) follows immediately from 2) of (5.1).  To prove 2) and 3) we consider the connected Lie subgroup G of $GL(n,\br)$ whose Lie algebra is $\fG$.  Suppose that  $\fZ(\fG)~\cap~\fP \neq \{0 \}$ and let X be a nonzero element of $\fZ(\fG)~\cap~\fP$.  The connected group G is generated by exp($\fG$) and hence G commutes with X.  It follows that G leaves invariant every eigenspace of X, and we conclude that $X = \lambda~I$ for some nonzero real number $\lambda$ since G acts irreducibly on $\br^{n}$.  This proves 2).

	We prove 3).  Suppose that  $\fZ(\fG)~\cap~\fK \neq \{0 \}$ and let A be a nonzero element of $\fZ(\fG)~\cap~\fK$.   Then $A^{2}$ is symmetric and negative semidefinite, and G commutes with $A^{2}$ since G commutes with A by the argument used in the proof of 2).  This argument also shows that $A^{2} = \lambda~I$ for some nonzero real number $\lambda$.  To show that dim $\fZ(\fG)~\cap~\fK = 1$ it suffices to prove that if $A^{2} = B^{2}$ for elements A,B of $\fZ(\fG)~\cap~\fK$, then either $A = B$ or $A = - B$.  If $0 = A^{2} - B^{2} = (A-B)(A+B)$, then $0 = (A-B)^{2}(A+B)^{2} = \lambda_{1} \lambda_{2}~I$, where $(A-B)^{2} = \lambda_{1}~I$ and $(A+B)^{2} = \lambda_{2}~I$.  Hence either $\lambda_{1} = 0$ and A $=$ B or $\lambda_{2} = 0$ and A $= -$ B.
		
	If J is a nonzero element of $\fZ(\fG)~\cap~\fK$, then by the discussion above we may assume that $J^{2} = -~I$, multiplying J by a suitable constant.  This proves 3).
\end{proof}
	
\begin{proposition}  Let G be a closed, connected, self adjoint, noncompact subgroup of $GL(n,\br)$ that acts irreducibly on $\br^{n}$.  Let $\fG$ denote the Lie algebra of G, and suppose that $\fG$ is not semisimple.  Then at least one of the following holds.

	1)  0 $\in \overline{G(v)}$ for all v $\in \br^{n}$.
	
	2) a)  There exists J $\in \fs \fo(n,\br)$ such that $J^{2} = - I$ and $\fZ(\fG) = \br-\rm{span}~\{J \}$.
	
\hspace{.10in}	b)  $G = Z \cdot G_{0}$, where $G_{0}$ is a connected, normal, semisimple subgroup of G and Z $\approx S^{1}$ is a compact, connected, central subgroup of G that lies in $O(n,\br)$.

	Moreover, if v is any vector in $\br^{n}$, then G(v) is closed in $\br^{n} \Leftrightarrow$ G$_{0}$(v) is closed in $\br^{n}$.
\end{proposition}

$\mathit{Remark}$

	If condition 2) holds, then n is even and $\br^{n}$ has a complex structure given by $(a+ib)v = av + i(Jv)$.  The group G becomes a group of complex linear maps of $\br^{n}$ since G commutes with J.

\begin{proof}  If $\fG$ is not semisimple, then $0 \neq \fZ(\fG) = \{\fZ(\fG)~\cap~\fK) \} \oplus  \{\fZ(\fG)~\cap~\fP) \}$ by (5.2).  We show that $\fZ(\fG)~\cap~\fP \neq \{0 \}$ implies 1) while $\fZ(\fG)~\cap~\fP = \{0 \}$ implies 2).

	Suppose first that $\fZ(\fG)~\cap~\fP \neq \{0 \}$, and let X $\in \fZ(\fG)~\cap~\fP$.  By 2) of (5.2) $X = \lambda~I$ for some nonzero real number $\lambda$.  If v is any vector in $\br^{n}$, then $exp(tX)(v) = exp(t \lambda)v \rightarrow 0$ as $t \rightarrow + \infty$ or as $t \rightarrow - \infty$.  It follows that $0 \in \overline{G(v)}$ for all v $\in \br^{n}$.
	
	Next suppose that $\fZ(\fG)~\cap~\fP = \{0 \}$, which implies that $\fZ(\fG) \subset \fK \subset \fs \fo(n,\br)$.  Assertion 2a) is assertion 3) of (5.2).
	
	Let $Z = exp(\fZ(\fG))$ and let $G_{0}$ be the connected Lie subgroup of $GL(n,\br)$ with Lie algebra $\fG_{0}$ in the notation of (5.1).  Since $J^{2} = - I$ it follows that $exp(tJ) = (cos~t)~I + (sin~t)~J$ for all $t \in \br$.  Hence $Z = exp(\br J)$ is a compact, connected 1-dimensional subgroup of $O(n,\br)$ by 2a).  The group $G_{0}$ is a normal subgroup of G since $\fG_{0}$ is an ideal of $\fG$ by 2) of (5.1). The groups G and $Z \cdot G_{0}$ are equal since both are connected subgroups of $GL(n,\br)$ with Lie algebra $\fG$.  This proves 2b) and the remark that follows it.
\end{proof}

$\mathit{The~KP~decomposition}$
	
	We are now ready to prove the main result of this section, the KP decomposition of G.  This result is well known in the context of algebraic groups, not necessarily connected, but the proof is different from that given here.  See for example Lemma 1.7 of [B-HC].  Let P $= exp(\fP) \subset$ G.
	
\begin{proposition}  Let G be a connected, closed, self adjoint, noncompact subgroup of $GL(n,\br)$, and let $\fG = \fK \oplus \fP$ denote the Lie algebra of G.  Let $K = G~\cap~O(n,\br)$.  Then  for every g $\in$ G there exist unique elements k $\in$ K and X $\in \fP$ such that g $=$ k exp(X) and
$|X| = d_{L}(g,K)$.
\end{proposition}

\begin{proof}  We prove uniqueness first.  Let g $=$ kp, where k $\in$ K and p $= exp(X) \in$ P for some X $\in \fP$.  Then $g^{t}g = p^{t} k^{t} k p = p^{2}$.  The elements of P are symmetric since the elements of $\fP$ are symmetric, and they are positive definite since the eigenvalues of exp(X) are the exponentials of the eigenvalues of X for all X $\in \fP$.  Hence p is the unique positive definite square root of $g^{t}g$.  The element k is also uniquely determined since $k = gp^{-1}$.

	We prove the existence of the decomposition g = kp.  Naively, we could try using the idea of the uniqueness proof and define p to be the positive definite square root of $g^{t}g$ and k to be $gp^{-1}$.  It is easy to check that k $\in O(n,\br)$, but it is not clear that p $\in$ G and k $\in$ G.  We must proceed more indirectly.

	 Let $d_{L}$ denote the canonical left invariant Riemannian metric on G.   Let g $\in$ G be given.  The subgroup K is a compact submanifold of G, and hence there exists a point k $\in$ K such that $d_{L}(g,k) \leq d_{L}(g,k')$  for all $k' \in$ K.  Since $d_{L}$ is left invariant it follows that $d_{L}(k^{-1}g, I) \leq d_{L}(k^{-1}g, k')$ for all $k' \in$ K.  Since (G,d$_{L}$) is complete by (3.1) the theorem of Hopf-Rinow states that there exists a geodesic $\gamma : [0,1] \rightarrow$ G such that $\gamma(0) = I, \gamma(1) = k^{-1}g$ and $d_{L}(I, k^{-1}g)$ is the length of $\gamma$. If X $= \gamma'(0) \in T_{I}G \approx \fG$, then X is orthogonal to $T_{I}K \approx \fK$ since I is the point in K closest to k$^{-1}$g (cf. Proposition 1.5 of [CE]). Hence X $\in \fP$ by 2) of (4.2) and 1) of (4.5).  By 3) of (4.5) the curve $\sigma(s) = exp(sX)$ is also a geodesic of G since $X^{t} = X$.
Hence $\gamma(s) = \sigma(s) = exp(sX)$ for all s since $\gamma'(0) = \sigma'(0) = X$.  It follows that $exp(X) = \gamma(1) = k^{-1}g$, or equivalently, $g = k exp(X)$.
	
	Finally, it was shown above that $|X| =$ length of $\gamma = d_{L}(k^{-1}g, K) = d_{L}(g,K)$.
\end{proof}

\begin{corollary}  Let G be as in (5.4), and let $d_{L}$ denote the canonical left invariant Riemannian metric on G. Let X $\in \fP$ with $|X| = 1$, and let $\gamma_{X}(s) = exp(sX)$ for s $\in \br$.  Then $\gamma_{X}(s)$ is a minimizing geodesic of $(G,d_{L})$ ; that is $d_{L}(\gamma_{X}(s), \gamma_{X}(t)) = |t-s|$, the length of $\gamma_{X}$ on [s,t], for all s $\leq$ t $\in \br$.
\end{corollary}

$\mathit{Remark}$  The same result holds for $(G, d_{R})$ since by (4.1)we have $d_{R}(\gamma_{X}(s), \gamma_{X}(s')) = \newline d_{L}(\gamma_{X}(s)^{t}, \gamma_{X}(s')^{t}) = d_{L}(\gamma_{X}(s), \gamma_{X}(s')) = |s - s'|$ for all $s \leq s' \in \br$.

\begin{proof}  We now prove (5.5).  For a fixed real number s write $\gamma_{X}(s) = k exp(Y)$, where $k = I$ and $Y = sX$.  The uniqueness part of the KP decomposition and the proof of (5.4) show that $|s| = |Y| = d_{L}(k^{-1} \gamma_{X}(s), I) = d_{L}(exp(sX), I)$. The curve $\gamma_{X}(s)$ is a geodesic of G by 3) of (4.5), and the left invariance of $d_{L}$ shows that $d_{L}(\gamma_{X}(s), \gamma_{X}(t)) = d_{L}(exp(sX), exp(tX)) = d_{L}(I, exp((t-s)X)) = |s-t|$.
\end{proof}

	The next result is an extension of the KP decomposition that is used in the proof of the main result, (8.1).  It also has some interest in its own right.  It is unclear if the elements k, X and h that appear in the statement of this result are unique.

    For the next result we define $\widetilde{\fP_{v}} = (\fK + \fG_{v})^{\perp} = \{\zeta \in \fG : \langle \zeta , \eta \rangle = 0 \}$ for
every $\eta \in \fK + \fG_{v}$.  Note that $\widetilde{\fP_{v}} \subset \fP$ since $\langle \fK , \fP \rangle = 0$ by (4.2).
	
\begin{proposition}  Let G be as in (5.4), and let v $\in \br^{n}, v \neq 0$. For every g $\in$ G there exist elements k $\in$ K, h $\in G_{v}$ and $X \in \widetilde{\fP_{v}}$ such that  $g = k~ exp(X)~ h$ and $|X| = d_{R}   (g,K \cdot G_{v}) = \newline d_{R}(exp(X), K \cdot G_{v}) = d_{R}(exp(X), I)$
\end{proposition}

\begin{proof}  The set $K \cdot G_{v}$ is closed in G, and hence by the completeness of $d_{R}$  there exist elements k $\in$ K and h $\in G_{v}$ such that $d_{R}(g,kh) = d_{R}(g, K \cdot G_{v})$.  If $g' = k^{-1} g h^{-1}$, then by 4) of (4.5) we have $d_{R}(g' , I) = d_{R}(g , kh) = d_{R}(g, K \cdot G_{v}) = d_{R}(g', K \cdot G_{v})$.  Let $\gamma : [0,1] \rightarrow G$ be a geodesic from $I = \gamma(0)$ to $g' = \gamma(1)$ whose length is $d_{R}(g', I)$.

	Let $X = \gamma'(0) \in T_{I}G \approx \fG$.  We show first that X $\in \widetilde{\fP_{v}}$.  If k $\in K \subset K \cdot G_{v}$, then by the previous paragraph $d_{R}(g',I) \leq d_{R}(g',k)$.  Hence X is orthogonal to $T_{I}K$  since K is a closed submanifold of G.  Identifying $T_{I}G$ with $\fG = \fK \oplus \fP$ and $T_{I}K$ with $\fK$ it follows from 2) of (4.2) that X $\in \fP$.  Similarly, $G_{v}$ is a closed submanifold of G, and the argument above shows that X is orthogonal to $T_{I} G_{v}$.  Identifying $T_{I} G_{v}$ with $\fG_{v}$ we conclude that X $\in \widetilde{\fP_{v}}=
(\fK + \fG_{v})^{\perp}$.
	
	Now let $\sigma(s) = exp(sX)$, where X is as above.  The curve $\sigma(s)$ is a geodesic of G by 3) of (4.5) since $X^{t} = X$.  Hence $\gamma(s) = \sigma(s)$ for all s $\in \br$ since both geodesics have the same initial velocity X.  It follows that $g' = \sigma(1) = exp(X)$ and hence
$g = k g' h = k~exp(X)~h$, where X $\in \widetilde{\fP_{v}}, k \in K$ and $h \in G_{v}$. From the work above and (5.5) we obtain $|X| = d_{R}(exp(X), Id) = d_{R}(g' , I) = d_{R}(g', K \cdot G_{v})= d_{R}(g, K \cdot G_{v})= d_{R}(exp(X), K \cdot G_{v})$.  In the final two equalities we also use 4) of (4.5).

\end{proof}

$\mathit{Comparison~of~d_{R}(g,I)~and~ |g|}$

\begin{proposition}  Let G be a connected, closed, self adjoint, noncompact subgroup of $GL(n,\br)$.  Let $d_{R}$ denote the canonical right invariant Riemannian metric on G.  Let c be a positive constant such that $d_{R}(k,I) \leq c$ for all k $\in$ K.  Let g $\in$ G and write $g = k~ exp(X)$, where k $\in$ K and X $\in \fP$.  Let $\lambda_{max}$ denote the largest eigenvalue of X.  Then

 $n^{- \frac{1}{2}} exp(-c)~ exp(|X| - \lambda_{max}) \leq \frac{exp(d_{R}(g,I))}{|g|} \leq exp(c)~ exp(|X| - \lambda_{max})$
\end{proposition}

$\mathit{Remark}$  The quantity $|X| - \lambda_{max}$ is always nonnegative and examples show that it may be unbounded as $|X| \rightarrow \infty$, even if $\lambda_{max} > 0$.  Hence the inequalities above, although close to optimal, don't give a satisfactory relationship between d(g,I) and $|g|$ for g $\in$ G.

\begin{proof}  The assertion is an immediate consequence of the following statements :

	(1)  $|X| - c \leq d_{R}(g,I) \leq |X| + c$
	
	(2)  $exp(\lambda_{max}) \leq |g| \leq n^{\frac{1}{2}}~ exp(\lambda_{max})$
	
	Note that $d_{R}(g,I) \leq d_{R}(k~exp(X),k) + d_{R}(k,I) \leq d_{R}(exp(X),I) + c = |X| + c$ by 4) of (4.5) and (5.5).  This proves the second inequality of (1), and the first inequality of (1) has a similar proof.  Let $\lambda_{1},~...~, \lambda_{n}$ be the eigenvalues of X. Since $g = k~exp(X)$ we have $|g|^{2} = trace~g^{t}g = trace~exp(2X) = \sum_{i=1}^{n}~exp(2 \lambda_{i})$.  Assertion (2) now follows immediately.
\end{proof}

\section {Minimal vectors}
	
	For a more complete discussion of the material in this section, see section 4 of [RS].
	
	A vector v $\in\br^{n}$ is $\mathit{minimal}$ for the G action if $|g(v)| \geq |v|$ for all g $\in$ G. Let $\fM$ denote the set of vectors in $\br^{n}$ that are minimal for G.  Note that 0 is always minimal and $\fM$ is invariant under $K = G~\cap~O(n,\br)$ since $|v| = |k(v)|$ for all v $\in \br^{n}$ and all k $\in$ K.
	
\begin{lemma}  Let G be a connected, closed, self adjoint, noncompact subgroup of $GL(n,\br)$.  Let v be a nonzero minimal vector.  Let $G_{v} = \{g \in G : g(v) = v \}$, and let $\fG_{v}$ denote the Lie algebra of $G_{v}$.  Then

	1)  $G_{v}^{t} = G_{v}$.
	
	2)  $\fG_{v} = \fK_{v} \oplus \fP_{v}$, where $\fK_{v} = \fK~\cap~\fG_{v}$ and $\fP_{v} = \fP~\cap~\fG_{v}$.
\end{lemma}

\begin{proof}  1)  Let g $\in G_{v}$ be given, and write $g = kexp(X)$, where k $\in$ K and X $\in \fP$.  If $f_{X}(t) = |exp(tX)(v)|^{2}$, then $f_{X}(0) = |v|^{2} = |g(v)|^{2} = |exp(X)(v)|^{2} = f_{X}(1)$.  Now $f_{X}'(0) = 0$ since v is minimal, and $f_{X}''(t) = 4|Xexp(tX)(v)|^{2} \geq 0$ for all t $\in \br$.  Since $f_{X}(1) = f_{X}(0)$ it follows that $0 = f_{X}''(0) = 4|X(v)|^{2}$.  This proves that exp(X)(v) $=$ v, and it follows that $k(v) = g~exp(X)^{-1}(v) = v$.  Hence $g^{t} = exp(X)^{t} k^{t} = exp(X) k^{-1} \in G_{v}$.

	2)  Note that $\fG_{v}^{t} = \fG_{v}$ by 1).  The assertion now follows as in the proof of 1) of (4.5).
\end{proof}

$\mathit{Moment~map}$

    For a fixed nonzero vector v of $\br^{n}$ the map $X \rightarrow \langle X(v) , v \rangle$ is a linear functional on $\fG$, and hence there exists
a unique vector m(v) $\in \fG$ such that $\langle m(v) , X \rangle = \langle X(v) , v \rangle$ for all X $\in \fG$.  Here $\langle m(v) , X \rangle = trace~ m(v)X^{t}$ as usual.  The map m : V $\rightarrow \fG$ is called the $\mathit{moment~ map}$ determined by G.  Note that m takes its values in $\fP$ by 1) of (4.5) ; $\langle m(v), X \rangle = \langle X(v) , v \rangle = 0$ if X $\in \fK$ since X is skew symmetric.

    The next result follows immediately from (4.3) iii) of [RS].

\begin{proposition}Let G be a connected, closed, self adjoint, noncompact subgroup of $GL(n,\br)$.  Then

    1)  A nonzero vector v of $\br^{n}$ is minimal $\Leftrightarrow m(v) = 0$.

    2)  If a nonzero vector v of $\br^{n}$ is minimal, then $G(v)~\cap~\fM = K(v)$.
\end{proposition}

	We now relate minimal vectors to closed orbits of G in the next result and its converse in (10.2).
	
\begin{proposition}   Let G denote a closed, noncompact subgroup of $GL(n,\br)$, and let v $\in \br^{n}$.  If the orbit G(v) is closed in $\br^{n}$, then G(v) contains a minimal vector.
\end{proposition}

\begin{proof}  Let c $= inf \{|g(v)| : g \in G \}$, and let $\{g_{k} \}$ be a sequence in G such that $|g_{k}(v)| \rightarrow c$ as $k \rightarrow \infty$.  Since the sequence $\{g_{k}(v) \}$ is bounded there exists a vector w $\in \br^{n}$ such that $g_{k}(v) \rightarrow w$, passing to a subsequence if necessary.  By continuity $|w| = c$,  and  w $\in G(v)$ since G(v) is closed.  Hence w is a minimal vector in G(v).
\end{proof}

$\mathit{Remark}$  In the proof of the main result (8.1) the function $g \rightarrow d_{R}(g, K\cdot G_{v})$ plays a major role.  The next result shows the geometric significance of the set $K \cdot G_{v}$.
	
\begin{proposition}  Let v be a nonzero minimal vector, and let $F_{v} : G \rightarrow \br$ be the function given by $F_{v}(g) = |g(v)|^{2}$.  The minimum locus for $F_{v}$ is the set $K \cdot G_{v}$, which is also the set of critical points for $F_{v}$.
\end{proposition}

\begin{proof}  Equip G with the canonical right invariant Riemannian structure $\langle , \rangle$.  A routine computation shows that grad $F_{v}(g) = 2~(R_{g})_{*}~m(g(v))$, where m : $\br^{n} \rightarrow \fG \approx T_{I}G$ is the moment map. By 1) of (6.2) g is a critical point of $F_{v} \Leftrightarrow$ g(v) is a minimal vector.  By 2) of (6.2) g(v) is a minimal vector $\Leftrightarrow g \in K \cdot G_{v}$, which by inspection is contained in the minimum locus of $F_{v}$.
\end{proof}
	
	The next result is a useful companion to the main result, (8.1).
	
\begin{proposition}  Let G denote a closed, connected, noncompact subgroup of $GL(n,\br)$, and let v be a nonzero vector of $\br^{n}$.  Then the following are equivalent.

	1)  The orbit G(v) is bounded in $\br^{n}$.
	
	2)  The function $g \rightarrow d_{R}(g, K \cdot G_{v})$ is bounded on G.
	
	3)  $G = K \cdot G_{v}$
	
	4)  $X(v) = 0$ for any X $\in \fP$.

    5)  $\fG = \fK + \fG_{v}$.
\end{proposition}
	
\noindent If G has no nontrivial compact, connected, normal subgroups, then G fixes v if any one of the conditions above holds.

	See Appendix II for the proof of (6.5).
\newline

	Next we investigate the growth functions $\lambda^{-}(v)$ and $\lambda^{+}(v)$ that appear in the statement of the main result, (8.1).

\section{The growth exponents $\lambda^{-}(v)$ and $\lambda^{+}(v)$}

	In this section let G be a closed, connected, self adjoint, noncompact subgroup of $GL(n,\br)$

	For each X $\in \fP, \br^{n}$ is an orthogonal direct sum of the eigenspaces of X.  For a nonzero element v of $\br^{n}$ and a nonzero element X of $\fP$ we define $\lambda_{X}(v)$ to be the largest eigenvalue $\lambda$ of X for which v has a nonzero component in $V_{\lambda} = \{w \in \br^{n} : X(w) = \lambda w \}$.

	For a nonzero element v of $\br^{n}$ we define
	
	$\lambda^{-}(v) = inf~\{\lambda_{X}(v) : X \in \widetilde{\fP_{v}}~\rm{and}~|X| = 1 \}$
	
	  $\lambda^{+}(v) = sup~\{\lambda_{X}(v) : X \in \widetilde{\fP_{v}}~\rm{and}~|X| = 1 \}$

where $\widetilde{\fP_{v}} = (\fK + \fG_{v})^{\perp} \subset \fP$.
\newline
	
\noindent $\mathit{Remark}$

	If G(v) is unbounded, then $\widetilde{\fP_{v}} \neq 0$ by 5) of (6.5).  It follows by continuity that $\widetilde{\fP_{w}} \neq 0$ for all w in some neighborhood O of v in $\br^{n}$, and hence G(w) is unbounded for all w $\in$ O by (6.5).

	The next result gives a dynamical definition of $\lambda_{X}(v)$, and it suggests why the main result could be true.

\begin{proposition}  Let X and v be nonzero elements of $\fP$ and $\br^{n}$ respectively.  Then $\lambda_{X}(v) = lim~_{t \rightarrow \infty} \frac{log~|exp(tX)(v)|}{t}$.
\end{proposition}

\begin{proof}  Write $v = \sum_{i=1}^{N} v_{i}$, where $X(v_{i}) = \lambda_{i} v_{i}$ for some real numbers $\{\lambda_{1}, ... , \lambda_{N} \}$.  Choose i so that $v_{i} \neq 0$ and $\lambda_{i} = \lambda_{X}(v)$.  Then $exp(tX)(v) = \sum_{k=1}^{N} exp(t \lambda_{k})(v_{k})$, and it follows that
$|exp(tX)(v)|^{2} = \sum_{k=1}^{N} exp(2t \lambda_{k}) |v_{k}|^{2}$.  Hence $exp(2t \lambda_{X}(v)) |v_{i}|^{2} \leq |exp(tX)(v)|^{2} \leq exp(2t \lambda_{X}(v) |v|^{2}$, and the assertion now follows immediately.
\end{proof}

\begin{proposition}  Let X and v be nonzero elements of $\fP$ and $\br^{n}$ respectively, and let g $\in GL(n,\br)$ be an element such that $gX = Xg$.  Then $\lambda_{X}(v) = \lambda_{X}(g(v))$.
\end{proposition}

\begin{proof} Write $v = \sum_{i=1}^{N} v_{i}$, where $X(v_{i}) = \lambda_{i} v_{i}$ for some real numbers $\{\lambda_{1}, ... , \lambda_{N} \}$.  If g $\in GL(n,\br)$ is an element such that $gX = Xg$, then $X(g(v_{i})) = g(X(v_{i})) = \lambda_{i} g(v_{i})$ for $1 \leq i \leq N$.  The assertion of the lemma now follows since $g(v) = \sum_{i=1}^{N} g(v_{i})$.
\end{proof}

\begin{proposition}  Let v be a nonzero vector of $\br^{n}$.  Then $\lambda^{-}(v) = \lambda^{-}(k(v))$ for every k $\in K = G \cap O(n,\br)$.
\end{proposition}

	For k $\in$ K the map $X \rightarrow k X k^{-1}$ is a linear isometry of $\fG$ by 1) of (4.2).  Moreover, $\fK + \fG_{k(v)} = k (\fK +\fG_{v}) k^{-1}$, and this implies that $\widetilde{\fP_{k(v)}} = k (\widetilde{\fP_{v}}) k^{-1}$.  The assertion now follows immediately from the next result

\begin{lemma}  $\lambda_{X}(v) = \lambda_{k X k^{-1}}(k(v))$ for all k $\in$ K, all nonzero X $\in \fP$ and all nonzero v $\in \br^{n}$.
\end{lemma}

\begin{proof}  Given X $\in \fP$, v $\in \br^{n}$ and k $\in$ K we write $v = \sum_{i=1}^{N} v_{i}$, where $X(v_{i}) = \lambda_{i} v_{i}$ for some real numbers $\{\lambda_{1}, ... , \lambda_{N} \}$.   Then $k(v) = \sum_{i=1}^{N} k(v_{i})$ and $(kXk^{-1})(k(v_{i})) = \lambda_{i} k(v_{i})$ for $1 \leq i \leq N$.  The assertion of the lemma now follows from the definitions of $\lambda_{X}(v)$ and $\lambda_{kXk^{-1}}(k(v))$.
\end{proof}

\begin{corollary}  Let v be a nonzero element of $\br^{n}$ such that G(v) is unbounded.  Then $\lambda^{-}(g(v)) \leq \lambda^{+}(v)$ for all g $\in$ G.
\end{corollary}

\begin{proof}  Let g $\in$ G be given.  By (5.6) we may write $g = k exp(X) h$, where k $\in$ K, h $\in G_{v}$ and X $\in \widetilde{\fP_{v}}$.  Then $g(v) = k exp(X)(v)$, and by (7.3) $\lambda^{-}(g(v)) = \lambda^{-}(exp(X)(v))$.  From (7.2) it follows that $\lambda^{-}(exp(X)(v)) \leq \lambda_{X}(exp(X)(v)) = \lambda_{X}(v) \leq \lambda^{+}(v)$.
\end{proof}

\begin{proposition}   Let G be a closed, connected, self adjoint, noncompact subgroup of $GL(n,\br)$, and let v $\in \br^{n}$ be a nonzero minimal vector such that G(v) is unbounded.  Then

	$0 < \lambda_{X}(v) \leq 1$ for all X $\in \widetilde{\fP_{v}}$ with $|X| = 1$.
\end{proposition}
	
\begin{proof}  Let X $\in \widetilde{\fP_{v}}$ with $|X| =1$ be given.  Since v is minimal we have $0 = X(|v|^{2}) = 2 \langle X(v) , v \rangle$.  Let $\lambda_{1}, ... , \lambda_{N}$ be the distinct eigenvalues of X and write $v = \sum_{i=1}^{N} v_{i}$, where $X(v_{i}) = \lambda_{i} v_{i}$ for $1 \leq i \leq N$.  Then $ 0 = \langle X(v) , v \rangle = \sum_{i=1}^{N} \lambda_{i} |v_{i}|^{2}$.  Observe that $X(v) \neq 0$ since $X \in \widetilde{\fP_{v}} \subset \fP_{v}^{\perp}$.  It follows that $\lambda_{i} > 0$ for some i with $v_{i} \neq 0$, and hence $\lambda_{X}(v) \geq \lambda_{i} > 0$.

	Let $X \in \fP$ with $|X| = 1$, and let $\{\lambda_{1}, ... , \lambda_{n} \}$ be the eigenvalues of X.  It follows that $\lambda_{i}^{2} \leq \sum_{k=1}^{n} \lambda_{k}^{2} = |X|^{2} = 1$ for any i with $1 \leq i \leq n$.  In particular $\lambda_{X}(v) \leq 1$.
\end{proof}

\begin{proposition}  Let G be a closed, connected, self adjoint, noncompact subgroup of $GL(n,\br)$, and let v $\in \br^{n}$ be a nonzero vector with G(v) unbounded.
	
	1)  Let X be a nonzero element of $\fP$.  Then for every $\epsilon > 0$ there exist neighborhoods U $\subset \br^{n}$ of v and O $\subset  \fP$ of X such that if $(X',v') \in O \times U$, then $\lambda_{X'}(v') \geq \lambda_{X}(v) - \epsilon$.
	
	2)  There exists X $\in \widetilde{\fP_{v}}$ with $|X| = 1$ such that $\lambda^{-}(v) = \lambda_{X}(v)$.
\end{proposition}

\begin{proof}

	We prove 1). Let a nonzero element X of $\fP$ be given.  We suppose that the assertion fails for some $\epsilon > 0$.  Then there exist sequences $\{X_{k} \} \subset \fP$ and $\{v_{k} \} \subset \br^{n}$ such that $v_{k} \rightarrow v, X_{k} \rightarrow X$ and $\lambda_{X_{k}}(v_{k}) < \lambda_{X}(v) - \epsilon$ for all k.  Passing to a subsequence we obtain the following properties simultaneously :

	a)  There exists a positive integer N such that each $X_{k}$ has N distinct eigenvalues $\{\lambda_{1}^{(k)}, ... , \lambda_{N}^{(k)} \}$.
	
	b)  There exist real numbers $\{\lambda_{1}, ... , \lambda_{N} \}$ such that $\lambda_{i}^{(k)} \rightarrow \lambda_{i}$ as k $\rightarrow \infty$ for $1 \leq i \leq N$.
	
	c)  There exist positive integers $\{m_{1}, ... , m_{N} \}$ and subspaces $\{V_{1}^{(k)}, ... , V_{N}^{(k)} \}$ of $\br^{n}$ such that
	
\hspace{.2in}  i)  dim $V_{i}^{(k)} = m_{i}$ for all k and for $1 \leq i \leq N$.

\hspace{.2in}  ii)  $X_{k} = \lambda_{i}^{(k)}~Id$ on $V_{i}^{k}$ for all k and for $1 \leq i \leq N$.

\hspace{.2in}  iii)  $\br^{n} = V_{1}^{(k)} \oplus ... \oplus V_{N}^{(k)}$, orthogonal direct sum, for all k.

	The existence of  $\{\lambda_{1}, ... , \lambda_{N} \}$ in b) follows from the fact that for $1 \leq i \leq N$ we have $|\lambda_{i}^{(k)}| \leq |X_{k}| \rightarrow |X|$ as k $\rightarrow \infty$.
	
	Passing to a further subsequence there exist subspaces $\{V_{i} \} \in G(m_{i},n)$, the (compact) Grassmannian of $m_{i}$ dimensional subspaces of $\br^{n}$, such that
	
	d)  $V_{i}^{(k)} \rightarrow V_{i}$ as k $\rightarrow \infty$ for $1 \leq i \leq N$.
	
	From c) and d) we obtain
	
	e)  $\br^{n} = V_{1} \oplus ... \oplus V_{N}$, orthogonal direct sum, and $X = \lambda_{i}$ Id on $V_{i}$ for $1 \leq i \leq N$.
	
	Note that the eigenvalues $\{\lambda_{1}, ... , \lambda_{N} \}$ for X may not all be distinct.
\newline
	By e) we may write $v = \sum_{i=1}^{N} v_{i}$, where $v_{i} \in V_{i}$ for $1 \leq i \leq N$. Choose $\beta, 1 \leq \beta \leq N$ such that $\lambda_{\beta} = \lambda_{X}(v)$.  The choice of $\beta \in \{1, ... , N \}$ may not be unique, but we show next that $\beta$ may be chosen so that v has a nonzero component in $V_{\beta}$.  Define $S_{\beta} = \{\alpha : 1 \leq \alpha \leq N~\rm{and}~\lambda_{\alpha} = \lambda_{\beta} \}$.  Then $V_{\beta}' = \oplus_{\alpha \in S_{\beta}} V_{\lambda_{\alpha}}$ is the $\lambda_{\beta}$ - eigenspace for X.  By the definition of $\lambda_{\beta} = \lambda_{X}(v)$ the vector v has a nonzero component in $V_{\beta}'$, and hence v has a nonzero component in $V_{\lambda_{\alpha}}$ for some $\alpha \in S_{\beta}$.  Replacing the original $\beta$ by this $\alpha$ we may now assume that $\beta$ has been chosen so that v has a nonzero component in $V_{\beta}$.
	
	 Since $v_{k} \rightarrow v$ as k $\rightarrow \infty$ it follows from d) that there exists a positive integer $K_{0}$ such that $v_{k}$ has a nonzero component in $V_{\beta}^{(k)}$ for all $k \geq K_{0}$.  By c), ii) we know that $X_{k} = \lambda_{\beta}^{(k)}$ Id on $V_{\beta}^{(k)}$, and hence $\lambda_{X_{k}}(v_{k}) \geq \lambda_{\beta}^{(k)}$ for all $k \geq K_{0}$.  It follows that $\lambda_{\beta}^{(k)} \leq \lambda_{X{k}}(v_{k}) < \lambda_{X}(v) - \epsilon$ for all $k \geq K_{0}$.  From b) we obtain $\lambda_{\beta} \leq \lambda_{X}(v) - \epsilon$, but this contradicts the fact that $\lambda_{\beta} = \lambda_{X}(v)$.  This completes the proof of 1) of (7.7).

	We prove 2) of (7.7).  Let $\{X_{k} \}$ be a sequence in $\widetilde{\fP_{v}}$ such that $|X_{k}| = 1$ for all k, and $\lambda_{X_{k}}(v) \rightarrow \lambda^{-}(v)$ as k $\rightarrow \infty$.  Passing to a subsequence if necessary, there exists X $\in \widetilde{\fP_{v}}$ with $|X| =1$ such that $X_{k} \rightarrow X$ as k $\rightarrow \infty$.  Let $\epsilon > 0$ be given.  By 1) of (7.7) there exists a positive integer $K_{0}$ such that $\lambda_{X_{k}}(v) \geq \lambda_{X}(v) - \epsilon$ for all k $\geq K_{0}$.  It follows that $\lambda^{-}(v) \geq \lambda_{X}(v)$ since $\epsilon > 0$ was arbitrary.  On the other hand $\lambda_{X}(v) \geq \lambda^{-}(v)$ by the definition of $\lambda^{-}(v)$.
\end{proof}

\begin{corollary}  Let G be a closed, connected, self adjoint, noncompact subgroup of $GL(n,\br)$, and let v $\in \br^{n}$ be a nonzero vector with G(v) unbounded. If $\lambda^{-}(v) < 0$, then $\overline{G(v)}$ contains the zero vector.
\end{corollary}

\begin{proof}  By 2) of (7.7) there exists X $\in \widetilde{\fP_{v}}$ such that $|X| = 1$ and $\lambda_{X}(v) = \lambda^{-}(v) < 0$.  Then exp(tX)(v) $\rightarrow 0$ by (7.1) or the proof of (7.1).
\end{proof}

$\mathit{Comparison~ between~ \lambda^{-}(v)~ and~ the~ Hilbert-Mumford~function~ M(v)}$

	We prove (2.6).  Recall that from the definitions in section 2 that we have  $\lambda_{X}(v) = - \mu_{-X}(v)$ for all nonzero X in $\fP$ and all nonzero v $\in \br^{n}$.

\begin{proposition}  Let v be a nonzero vector of $\br^{n}$.  Then $\lambda^{-}(v) \geq - M(v)$.
\end{proposition}

\begin{proof}  By 2) of (7.7) we may choose X $\in \widetilde{\fP_{v}}$ such that $|X| = 1$ and $\lambda^{-}(v) = \lambda_{X}(v)$.  Then $-\lambda^{-}(v) = - \lambda_{X}(v) = \mu_{-X}(v) \leq M(v)$.
\end{proof}

\section{Proof of the main result}
	
	We now reach the main result, whose proof will be completed after the proof of (8.6).
		
\begin{theorem}  Let G be a closed, connected, self adjoint, noncompact subgroup of $GL(n,\br)$.  Let $d_{R}$ denote the canonical right invariant Riemannian metric on G.  Let v be a nonzero  vector such that the orbit G(v) is unbounded.  Then

	1)  The function $g \rightarrow d_{R}(g, G_{v})$ is unbounded on G.
	
	2)  If v is minimal, then $\lambda^{-}(v) > 0$.  For arbitrary nonzero v we have $-1 \leq \lambda^{-}(v) \leq \lambda^{+}(v) \leq 1$.

	3)  $\lambda^{-}(v) \leq \underline{lim} \hspace{.05in} _{d_{R}(g, G_{v}) \rightarrow \infty} \hspace{.1in} \frac{log|g(v)|}{d_{R}(g, G_{v})} \leq  \overline{lim} \hspace{.05in}  _{d_{R}(g, G_{v}) \rightarrow \infty} \hspace{.1in}\frac{log|g(v)|}{d_{R}(g, G_{v})} \leq \lambda^{+}(v)$
\end{theorem}

    The first assertion follows from (6.4) while the second assertion follows from (7.6) and the definitions of $\lambda^{-}(v)$ and $\lambda^{+}(v)$.
Assertion 3) will follow from (8.3) and (8.6).

      We begin the proof of 3).  We recall from the first remark after the statement of (2.1) that we can replace $d_{R}(g, G_{v})$ by $d_{R}(g, K \cdot G_{v})$ in the statement of (8.1).  For the remainder of the proof of 3) we make this replacement.
    \newline	

$\mathit{The~subset~Q_{v}~in~\widetilde{\fP_{v}}}$
\newline

	For a nonzero element v of $\br^{n}$ let $Q_{v} = \{X \in \widetilde{\fP_{v}} : |X| =1~\rm{and}~d_{R}(exp(sX), Id) = d_{R}(exp(sX) , K \cdot G_{v})~\rm{for~all}~ s \geq 0 \}$.  In words, a unit vector X in $\widetilde{\fP_{v}}$ lies in $Q_{v} \Leftrightarrow$ the identity is the point
on $K \cdot G_{v}$ closest to exp(X)(s) for all s $> 0$.
	
\begin{proposition}  Let G be a closed, connected, self adjoint, noncompact subgroup of $GL(n,\br)$.  Let v be a nonzero vector such that G(v) is unbounded.  Then

	1)  $Q_{v}$ is nonempty.
	
	2)  $\underline{lim}~_{d_{R}(g, K \cdot G_{v}) \rightarrow \infty}~\frac{log|g(v)|}{d_{R}(g, K \cdot G_{v})}  \leq \lambda_{X}(v)$ for all $X \in Q_{v}$, with equality for some X $\in Q_{v}$.
\end{proposition}

	As an immediate consequence we obtain
	
\begin{corollary}  Let G be a closed, connected, self adjoint, noncompact subgroup of $GL(n,\br)$.  Let v be a nonzero vector such that G(v) is unbounded. Then $\lambda^{-}(v) \leq \underline{lim}~_{d_{R}(g, K \cdot G_{v}) \rightarrow \infty}~\frac{log|g(v)|}{d_{R}(g, K \cdot G_{v})}$.
\end{corollary}

	We now begin the proof of (8.2).  Assertion 1) will follow from the next result.

\begin{lemma}  $\underline{lim}~_{d_{R}(g, K \cdot G_{v}) \rightarrow \infty}~\frac{log|g(v)|}{d_{R}(g, K \cdot G_{v})}  \geq inf\{\lambda_{X}(v) : X \in Q_{v} \}$
\end{lemma}

\begin{proof}  Let $\{g_{r} \}$ be any sequence in G such that $d_{R}(g_{r}, K \cdot G_{v}) \rightarrow \infty$ as r $\rightarrow \infty$.  It suffices to prove that $A = \underline{lim}~_{r \rightarrow \infty}~\frac{log|g_{r}(v)|}{d_{R}(g_{r}, K \cdot G_{v})} \geq inf \{\lambda_{X}(v) : X \in Q_{v} \}$.  If $A = \infty$ there is nothing to prove, so we assume that A is finite.  Passing to a subsequence we may assume that $\frac{log|g_{r}(v)|}{d_{R}(g_{r}, K \cdot G_{v})} \rightarrow A$ as r $\rightarrow \infty$.

	By (5.6) we may write $g_{r} = k_{r} exp(X_{r}) h_{r}$ for each r, where $k_{r} \in K, h_{r} \in G_{v}$ and $X_{r} \in \widetilde{\fP_{v}}$ with $|X_{r}| = d_{R}(g_{r}, K \cdot G_{v})= d_{R}(exp(X), K \cdot G_{v}) = d_{R}(exp(X), I)$.  Write $X_{r} = t_{r} Y_{r}$, where $Y_{r} = \frac{X_{r}}{|X_{r}|}$ and $ t_{r} = |X_{r}| \rightarrow \infty$ as r $\rightarrow \infty$.  Passing to a further subsequence there exists Y $\in \widetilde{\fP_{v}}$ with $|Y| = 1$ such that $Y_{r} \rightarrow Y$ as r $\rightarrow \infty$.
	
	Let N, $\{\lambda_{1}^{(r)}, ... , \lambda_{N}^{(r)} \}, \{\lambda_{1}, ... , \lambda_{N} \}, \{V_{1}^{(r)}, ... , V_{N}^{(r)} \}$ and $\{V_{1}, ... , V_{N} \}$ have the same definition and properties of a), b), c) d) and e) of the proof of 1) of (7.7), with $Y_{r}$ replacing $X_{k}$ and v$_{r} =$ v for all r.  Write $v = \sum_{i=1}^{N} v_{i}$, where $Y(v_{i}) = \lambda_{i} v_{i}$ for $1 \leq i \leq N$.  Choose i such that $v_{i} \neq 0$ and $\lambda_{i} = \lambda_{Y}(v)$.  Then by d) there exists a positive integer $R_{0}$ such that v has a nonzero component $v_{i}^{(r)}$ in $V_{i}^{(r)}$ for all $r \geq R_{0}$.  By c) we have for each r the expression $v = \sum_{j=1}^{N} v_{j}^{(r)}$, where $Y_{r}(v_{j}^{(r)}) = \lambda_{j}^{(r)} v_{j}^{(r)}$ for $1 \leq j \leq N$ and all r.  Hence $|g_{r})v)|^{2} = |exp(X_{r})(v)|^{2} = |exp(t_{r} Y_{r})(v)|^{2} = \sum_{j=1}^{N} exp(2 t_{r} \lambda_{j}^{(r)}) |v_{j}^{(r)}|^{2} \geq exp(2 t_{r} \lambda_{i}^{(r)}) |v_{i}^{(r)}|^{2}$.  It follows that $\frac{log |g_{r}(v)|}{d_{R}(g_{r}, K \cdot G_{v})} = \frac{log |g_{r}(v)|}{t_{r}} \geq \lambda_{i}^{(r)} + \frac{log |v_{i}^{(r)}|}{t_{r}}$.  Since $v_{i}^{(r)} \rightarrow v_{i} \neq 0$ and $\lambda_{i}^{(r)} \rightarrow \lambda_{i} = \lambda_{Y}(v)$ as r $\rightarrow \infty$ we obtain $A  = lim~_{r \rightarrow \infty} \frac{log|g_{r}(v)|}{d_{R}(g_{r}, K \cdot G_{v})} \geq \lambda_{Y}(v)$.
	
	To complete the proof of (8.4) it remains only to prove that $Y \in Q_{v}$.  By construction $|Y| = 1$.  It will be useful to note the following consequence of the triangle inequality : Let  $\gamma(t)$ be a unit  speed geodesic of $(G, d_{R})$ such that $\gamma(0) = Id$ and $d_{R}(\gamma(t_{0}), Id) = d_{R}(\gamma(t_{0}), K \cdot G_{V})$ for some $t_{0} > 0$.  Then $d_{R}(\gamma(s), Id) = d_{R}(\gamma(s), K \cdot G_{v})$ for $0 \leq s \leq t_{0}$.
	
	From the discussion above we recall that $X_{r} = t_{r} Y_{r}$, where $t_{r} = |X_{r}|$.  By the remark following (5.5) and the definition of $X_{r}$ we know that $t_{r} = d_{R}(exp(t_{r} Y_{r}), Id) = d_{R}(exp(X_{r}), Id) = d_{R}(exp(X_{r}), K \cdot G_{v}) = d_{R}(exp(t_{r} Y_{r}), K \cdot G_{v})$.  Applying the observation of the previous paragraph to the unit speed geodesics $\gamma_{r}(s) = exp(s Y_{r})$ we find that $d_{R}(exp(s Y_{r}), Id) = \newline d_{R}(exp(s Y_{r}), K \cdot G_{v})$ for all r and $0 \leq s \leq t_{r}$.  Now $Y_{r} \rightarrow Y$ and $t_{r} \rightarrow \infty$ as r $\rightarrow \infty$, and we conclude that $d_{R}(exp(s Y), Id) = d_{R}(exp(s Y), K \cdot G_{V})$ for all s $\geq 0$.  Hence $Y \in Q_{v}$.
\end{proof}

	For the proof of (8.2) it remains to prove assertion 2).
	
\begin{lemma} For every X $\in Q_{v}~,~ \lambda_{X}(v) \geq \underline{lim}~_{d_{R}(g, K \cdot G_{v}) \rightarrow \infty}~\frac{log|g(v)|}{d_{R}(g, K \cdot G_{v})}$.
\end{lemma}

\begin{proof}  Let X $\in Q_{v}$ be arbitrary.  Since $|X| = 1$ it follows from the remark following (5.5) and the definition of $Q_{v}$ that $d_{R}(exp(tX), K \cdot G_{v}) = d_{R}(exp(tX), Id) = t$ for all t $> 0$.  By (7.1) it follows that $\underline{lim}~_{d_{R}(g, K \cdot G_{v}) \rightarrow \infty}~\frac{log|g(v)|}{d_{R}(g, K \cdot G_{v})} \leq lim~_{t \rightarrow \infty} \frac{log|exp(tX)(v)|}{d_{R}(exp(tX), K \cdot G_{v})} = lim~_{t \rightarrow \infty} \frac{log|exp(tX)(v)|}{t} = \lambda_{X}(v)$.
\end{proof}

	We complete the proof of assertion 2) of (8.2).  From (8.4) and (8.5) it follows that  \newline $\underline{lim}~_{d_{R}(g, K \cdot G_{v}) \rightarrow \infty}~\frac{log|g(v)|}{d_{R}(g, K \cdot G_{v})} = inf  \{\lambda_{X}(v) : X \in Q_{v} \}$.  Let $\{X_{r} \}$ be a sequence in $Q_{v}$ such that $\lambda_{X_{r}}(v) \rightarrow \underline{lim}~_{d_{R}(g, K \cdot G_{v}) \rightarrow \infty}~\frac{log|g(v)|}{d_{R}(g, K \cdot G_{v})}$ as r $\rightarrow \infty$.  Passing to a further subsequence we may assume that $X_{r} \rightarrow X$ as r $\rightarrow \infty$.  Note that $|X| = 1$ and $X \in Q_{v}$ since $Q_{v}$ is closed in $\widetilde{\fP_{v}}$.  By 1) of (7.7) it follows that  $\underline{lim}~_{d_{R}(g, K \cdot G_{v}) \rightarrow \infty}~\frac{log|g(v)|}{d_{R}(g, K \cdot G_{v})} = lim~_{r \rightarrow \infty}~\lambda_{X_{r}}(v) \geq \lambda_{X}(v)$.  Equality follows by (8.5).
	
\begin{proposition}  Let G be a closed, connected, self adjoint, noncompact subgroup of $GL(n,\br)$.  Let v be a nonzero vector such that G(v) is unbounded.  Then $\overline{lim}~_{d_{R}(g, K \cdot G_{v}) \rightarrow \infty}~\frac{log|g(v)|}{d_{R}(g, K \cdot G_{v})}  \leq \lambda^{+}(v).$
\end{proposition}

\begin{proof}  Let $\{g_{r} \}$ be a sequence in G such that $d_{R}(g_{r}, K \cdot G_{v}) \rightarrow \infty$ as r $\rightarrow \infty$. It suffices to prove that $\overline{lim}~_{r \rightarrow \infty}~\frac{log|g_{r}(v)|}{d_{R}(g_{r}, K \cdot G_{v})} \leq \lambda^{+}(v)$.  As in the proof of (8.4) we may write $g_{r} = k_{r} exp(t_{r} Y_{r}) h_{r}$, where $k_{r} \in K, h_{r} \in G_{v}, Y_{r} \in \widetilde{\fP_{v}}~\rm{with}~|Y_{r}| = 1$ and $t_{r} = d_{R}(g_{r}, K \cdot G_{v}) \rightarrow \infty$ as r $\rightarrow \infty$.  For each r we write $v = \sum_{i=1}^{N_{r}} v_{i}^{(r)}$, where $Y_{r}(v_{i}^{(r)}) = \lambda_{i}^{(r)} v_{i}^{(r)}$ for some real numbers $\{\lambda_{1}^{(r)}, ... , \lambda_{N_{r}}^{(r)} \}$.  Then $|g_{r}(v)|^{2} =|exp(t_{r} Y_{r})(v)|^{2} = \sum_{i=1}^{N_{r}} exp(2t_{r} \lambda_{i}^{(r)}) |v_{i}^{(r)}|^{2} \leq exp(2 t_{r} \lambda^{+}(v)) |v|^{2}$.  It follows that $\overline{lim}~_{r \rightarrow \infty}~\frac{log|g_{r}(v)|}{d_{R}(g_{r}, K \cdot G_{v})} = \overline{lim}~_{r \rightarrow \infty}~\frac{log|g_{r}(v)|}{t_{r}} \leq \lambda^{+}(v)$.
\end{proof}

    As we observed earlier, the proof of 3) of (8.1) now follows from (8.3) and (8.6).  The proof of (8.1) ia complete.
\newline

$\mathit{The~case~that~G_{v}~is~compact}$

\begin{corollary}  Let G be a closed, connected, self adjoint, noncompact subgroup of $GL(n,\br)$.  Let $d_{R}$ and $d_{L}$ denote respectively the canonical right invariant and left invariant Riemannian metrics on G.  Let v be a nonzero  vector such that the orbit G(v) is unbounded and $G_{v}$ is compact.  Let I denote the identity matrix in $GL(n, \br)$.  Then

	   $\lambda^{-}(v) = \underline{lim} \hspace{.05in} _{d_{R}(g, I) \rightarrow \infty} \hspace{.1in} \frac{log|g(v)|}{d_{R}(g, I)} = \underline{lim} \hspace{.05in} _{d_{L}(g, I) \rightarrow \infty} \hspace{.1in} \frac{log|g(v)|}{d_{L}(g, I)}\leq  \overline{lim} \hspace{.05in}  _{d_{L}(g, I) \rightarrow \infty} \hspace{.1in}\frac{log|g(v)|}{d_{L}(g, I)} = \overline{lim} \hspace{.05in}  _{d_{R}(g, I) \rightarrow \infty} \hspace{.1in}\frac{log|g(v)|}{d_{R}(g, I)} = \lambda^{+}(v)$
\end{corollary}

\begin{proof}

	We first prove the following  weak inequalities for $\lambda^{-}(v)$ and $\lambda^{+}(v)$.
	
	(*)   $\lambda^{-}(v) \leq \underline{lim} \hspace{.05in} _{d_{R}(g, I) \rightarrow \infty} \hspace{.1in} \frac{log|g(v)|}{d_{R}(g, I)} = \underline{lim} \hspace{.05in} _{d_{L}(g, I) \rightarrow \infty} \hspace{.1in} \frac{log|g(v)|}{d_{L}(g, I)}\leq  \overline{lim} \hspace{.05in}  _{d_{L}(g, I) \rightarrow \infty} \hspace{.1in}\frac{log|g(v)|}{d_{L}(g, I)} = \overline{lim} \hspace{.05in}  _{d_{R}(g, I) \rightarrow \infty} \hspace{.1in}\frac{log|g(v)|}{d_{R}(g, I)} \leq \lambda^{+}(v)$
\newline
	
	Observe that $K \cdot G_{v}$ is compact, and hence there exists a positive constant c such that $d_{R}(I,\xi) \leq c$ and $d_{L}(I,\xi) \leq c$ for all $\xi \in K \cdot G_{v}$.  A routine argument with the triangle inequality yields

  	a)  $|d_{R}(g, K \cdot G_{v}) - d_{R}(g,I)| \leq c$ for all g $\in$ G.
	
	Let g $\in$ G be given.  By the KP decomposition of G there exist unique elements k $\in$ K and p $\in P = exp(\fP)$ such that $g = kp$.  We assert
	
	b) $|d_{R}(g,I) - d_{R}(p,I)| \leq c$ and  $|d_{L}(g,I) - d_{L}(p,I)| \leq c$
	
\noindent It will then follow immediately that

	c)  $|d_{R}(g,I) - d_{L}(g,I)| \leq |d_{R}(g,I) - d_{R}(p,I)| + |d_{R}(p,I) - d_{L}(p,I)| + |d_{L}(p,I) - d_{L}(g,I)| \leq 2c$
	
\noindent since $d_{R}(p,I) = d_{L}(p^{t},I^{t}) = d_{L}(p,I)$ by (4.1).  From a), c) and (8.1) the assertion (*) will follow immediately.

	To complete the proof of (*) it remains only to prove b).  Note that $d_{R}(p,I) \leq d_{R}(p,kp) + d_{R}(g,I) = d_{R}(I,k) + d_{R}(g,I) \leq c + d_{R}(g,I)$.  Similarly, $d_{R}(p,I) \geq - d_{R}(p,kp) + d_{R}(g,I) = - d_{R}(I,k) + d_{R}(g,I) \geq - c + d_{R}(g,I)$.  This proves the first inequality in b).  To prove the second inequality in b) observe that $d_{L}(g,I) \leq d_{L}(k,I) + d_{L}(k, kp) \leq c + d_{L}(I,p)$ and $d_{L}(g,I) \geq - d_{L}(k,I) + d_{L}(k, kp) \geq - c + d_{L}(I,p)$
\newline

	We complete the proof of  (8.7) by showing that the inequalities in (*) for $\lambda^{-}(v)$ and $\lambda^{+}(v)$ are actually equalities.
By 2) of (7.7) we may choose an element X of $\widetilde{\fP_{v}}$ such that $|X| = 1$ and $\lambda^{-}(v) = \lambda_{X}(v)$.  Let $g_{r} = exp(rX)$.  Then $d_{L}(g_{r}, I) =$ r for every integer r by (5.5).  Moreover, $\frac{log|g_{r}(v)|}{d_{L}(g_{r}, I)} = \frac{log|exp(rX)(v)|}{r} \rightarrow \lambda_{X}(v)) = \lambda^{-}(v)$ by (7.1).  This proves that $\lambda^{-}(v) = \underline{lim} \hspace{.05in} _{d_{L}(g, I) \rightarrow \infty} \hspace{.1in} \frac{log|g(v)|}{d_{L}(g, I)}$.

	We prove that the inequality for $\lambda^{+}(v)$ in (*) is sharp.  Let X in $\widetilde{\fP_{v}}$ with $|X| = 1$ be given, and define $g_{r} = exp(rX)$ for every positive integer r.  The argument above shows that $lim~_{r \rightarrow \infty} \frac{log|g_{r}(v)|}{d_{L}(g_{r}, I)} = \lambda_{X}(v)$.  Since $\lambda^{+}(v) = sup \{\lambda_{X}(v) : X \in \widetilde{\fP_{v}}, |X| =1 \}$ it follows that $\lambda^{+}(v) = \newline \overline{lim}~_{d_{L}(g, I) \rightarrow \infty}~ \frac{log|g(v)|}{d_{L}(g, I)}$.
\end{proof}

\section{Uniform growth on compact subsets of $\fM'$}

    Let O $= \{v \in \br^{n} : dim~G_{v} \leq dim~G_{w}~\rm{for~all}~w \in \br^{n} \}$.	It is well known that O is a nonempty G-invariant Zariski open subset O of $\br^{n}$.
	
\begin{proposition}  Let $v \in$ O be given.  For every number $\epsilon > 0$ there exists an open set $U \subset \br^{n}$ such that if $v' \in $ U  then $\lambda^{-}(v') \geq \lambda^{-}(v) - \epsilon$.
\end{proposition}

\begin{lemma}  Let v $\in$ O be given and let $\{v_{r} \} \subset O$ be a sequence converging to v.  Let $\{Y_{r} \}$ be a sequence in $\fP$ converging to a vector Y $\in \fP$ such that $Y_{r} \in \widetilde{\fP_{v_{r}}}$ for every r.  Then Y $\in \widetilde{\fP_{v}}$.
\end{lemma}

\begin{proof}  Let $\zeta \in \fK + \fG_{v}$ be given and write $\zeta = K + H$, where K $\in \fK$ and H $\in \fG_{v}$.  By the definition of O we know that dim $\fG_{v_{r}} =$ dim $\fG_{v}$ for all r, and hence there exists a sequence $\{H_{r} \}$ such that $H_{r} \in \fG_{v_{r}}$ for every r and $H_{r} \rightarrow H$ as r $\rightarrow \infty$.  Let $\zeta_{r} = K + H_{r} \in \fK + \fG_{v_{r}}$ for all r.  By definition $0 = \langle \zeta_{r} , Y_{r} \rangle$ for all r since $Y_{r} \in \widetilde{\fP_{v_{r}}}$.  Hence $\langle \zeta , Y \rangle = $ lim$_{r \rightarrow \infty} \langle \zeta_{r} , Y_{r} \rangle = 0$, which proves that $Y \in\widetilde{\fP_{v}}$.
 \end{proof}

	We now complete the proof of (9.1).  Suppose that this is false for some $\epsilon > 0$.  Then there exists a sequence $\{v_{r} \} \subset$ O such that $v_{r} \rightarrow v$ and $\lambda^{-}(v_{r}) < \lambda^{-}(v) - \epsilon$ for all r.  By 2) of (7.7) we may choose $Y_{r} \in \widetilde{\fP_{v_{r}}}$ with $|Y_{r}| = 1$ such that $\lambda^{-}(v_{r}) =  \lambda_{Y_{r}}(v_{r})$ for all r.  Let $Y_{r} \rightarrow Y \in \fP$, passing to a subsequence if necessary.  Then $|Y| = 1$ by continuity, and Y $\in \widetilde{\fP_{v}}$ by (9.2).  From 1) of (7.7) we obtain $\lambda^{-}(v) \leq \lambda_{Y}(v) \leq \underline{lim}~_{r \rightarrow \infty} \lambda_{Y_{r}}(v_{r}) =  \underline{lim}~_{r \rightarrow \infty} \lambda^{-}(v_{r}) \leq \lambda^{-}(v) - \epsilon$, which is a contradiction.
	
\begin{corollary}  Let C be a compact subset of O.  Then $\lambda^{-}$ has a minimum value on C.
\end{corollary}

\begin{proof}  Let $c = inf~\{\lambda^{-}(v') : v' \in C \}$, and let $\{v_{k} \} \subset C$ be a sequence such that $\lambda^{-}(v_{k}) \rightarrow c$ as k $\rightarrow \infty$.  By the compactness of C we may assume, passing to a subsequence if necessary,  that there exists a vector v $\in$ C such that $v_{k} \rightarrow v$ as k $\rightarrow \infty$.  Then $c = lim~_{k \rightarrow \infty}~\lambda^{-}(v_{k}) \geq \lambda^{-}(v)$ by (9.1).  Equality must hold by the definition of c.
\end{proof}

	Let $\fM'$ denote the set of minimal vectors in O.

\begin{proposition}  Let C be a compact subset of $\fM'$, and let $c > 0$ be the minimum value of $\lambda^{-}$ on C.  Then for every real number $c'$ with $0 < c' < c$ there exists a positive number $R_{0}$ such that if v $\in$ C, g $\in$ G and $d_{R}(g, G_{v}) > R_{0}$, then $\frac{log|g(v)|}{d_{R}(g, G_{v})} > c'$.
\end{proposition}

\begin{proof}

	By the first remark following (2.1) we may replace $d_{R}(g, G_{v})$ by $d_{R}(g, K \cdot G_{v})$ in the statement of (9.4).  We make this replacement in the remainder of the proof of (9.4).

    We first reduce to the case that every vector in C has length 1.  Let $C_{1} = \{\frac{v}{|v|} : v \in C \}$.   Note that $\lambda_{X}(rv) = \lambda_{X}(v), \lambda^{-}(rv) = \lambda^{-}(v)$ and $G_{rv} = G_{v}$ for every nonzero X $\in \fP, r \in \br$, and v $\in \br^{n}$.  By the compactness of C there exists b $> 0$ such that if $|rv| = 1$ for some r $\in \br$ and some v $\in$ C, then $|r| \leq b$.
	
	Let $c'$ be a positive number with $c' < c$ and choose $\epsilon > 0$ such that $c' + \epsilon < c$.  If (9.4) is true for the compact subset $C_{1} \subset \fM'$, then there exists $R_{0} > 0$ such that if $d_{R}(g, K \cdot G_{v_{1}}) \geq R_{0}$, then  $\frac{log|g(v_{1})|}{d_{R}(g, K \cdot G_{v_{1}})} \geq c' + \epsilon$ for every $v_{1} \in C_{1}$.  Make $R_{0}$ larger if necessary so that $|\frac{log(b)}{R_{0}}| < \epsilon$.
	
	Now let v $\in$ C and g $\in$ G be given so that $R_{0} \leq d_{R}(g, K \cdot G_{v}) = d_{R}(g, K \cdot G_{v_{1}})$, where $v_{1} = rv \in C_{1}$ and $r = \frac{1}{|v|} > 0$.  Then $c' + \epsilon \leq \frac{log|g(v_{1})|}{d_{R}(g, K \cdot G_{v_{1}})} = \frac{log|r|}{d_{R}(g, K \cdot G_{v})} + \frac{log|g(v)|}{d_{R}(g, K \cdot G_{v})} \leq \frac{log(b)}{R_{0}}$ + $\frac{log|g(v)|}{d_{R}(g, K \cdot G_{v})} < \epsilon + \frac{log|g(v)|}{d_{R}(g, K \cdot G_{v})}$.  Hence (9.4) holds for all v $\in$ C if it holds for all unit vectors $v_{1} \in C_{1}$.
\newline
	
	Henceforth we assume that all vectors v in C have length 1.

\begin{lemma}  Let $v \in \fM'$ with $|v| = 1$, and let X be an element of $\fP$ with $X(v) \neq 0$.  Let $\varphi_{X}(s) = \frac{log|exp(sX)(v)|}{s}$.  Then $\varphi_{X}'(s) \geq 0$ for $s > 0$.
\end{lemma}

\begin{proof}  Let $f_{X}(s) = |exp(sX)(v)|^{2}$.  Then $\varphi_{X}(s) = \frac{1}{2}~\frac{log(f_{X})(s)}{s}$ and $\varphi_{X}'(s) = (\frac{1}{2s^{2}})~(s~\frac{f_{X}'(s)}{f_{X}(s)} - log~f_{X}(s))$.  If $\lambda_{X}(s) = s~\frac{f_{X}'(s)}{f_{X}(s)} - log~f_{X}(s)$, then it suffices to show that $\lambda_{X}(s) > 0$ for $s > 0$.  Note that $\lambda_{X}(0) = - log(|v|^{2}) = 0$, so it suffices to show that $\lambda_{X}'(s) \geq 0$ for all $s > 0$.  We calculate $\lambda_{X}'(s) = s \cdot \frac{d}{ds}~(\frac{f_{X}'(s)}{f_{X}(s)}) = s \cdot \frac{f_{X}(s)~f_{X}''(s) - f_{X}'(s)^{2}}{f_{X}(s)^{2}} = \newline 4s \cdot \frac{|exp(sX)(v)|^{2}~|Xexp(sX)(v)|^{2} - \langle Xexp(sX)(v) , exp(sX)(v) \rangle ^{2}}{|exp(sX)(v)|^{4}} \geq 0$ for $s > 0$.
\end{proof}

	We now complete the proof of (9.4).  Suppose that the assertion of (9.4) is false for some positive number $c' < c$.  Then there exist sequences $\{v_{r} \} \subset C$ and $\{g_{r} \} \subset G$  such that $|v_{r}| = 1$ for all r,  $d_{R}(g_{r}, K \cdot G_{v_{r}}) \rightarrow \infty$ as $r \rightarrow \infty$ and $\frac{log|g_{r}(v_{r})|}{d_{R}(g, K \cdot G_{v_{r}})} \leq c'$ for all r.  By (5.6) there exist elements $k_{r} \in K, h_{r} \in G_{v_{r}}, t_{r} \in \br$ and $Y_{r} \in \widetilde{\fP_{v_{r}}}$ such that $g_{r} = k_{r} exp(t_{r} Y_{r}) h_{r}$, where $|Y_{r}| = 1$ and $t_{r} = d_{R}(g_{r}, K \cdot G_{v_{r}})$ for all r.  Let $Y_{r} \rightarrow Y \in \fP$, passing to a subsequence.  Then $|Y| = 1$ by continuity, and Y $\in \widetilde{\fP_{v}}$ by (9.2).  By the compactness of C there exists v $\in$ C such that $v_{r} \rightarrow v$ as $r \rightarrow \infty$, passing to a further subsequence if necessary.
	
	Fix a positive number s.  Then $ s \leq t_{r}$ for large r since $t_{r} = d_{R}(g_{r}, K \cdot G_{v_{r}}) \rightarrow \infty$ as $r \rightarrow \infty$.  By applying (9.5) to the functions $\varphi_{Y_{r}}(s)$ we obtain $\frac{log|exp(sY)(v)|}{s} =  lim~_{r \rightarrow \infty}~\frac{log|exp(sY_{r})(v_{r})|}{s}  \leq \underline{lim}~_{r \rightarrow \infty}~\frac{log|exp(t_{r}Y_{r})(v_{r})|}{t_{r}} =  \underline{lim}~_{r \rightarrow \infty}~\frac{log|g_{r}(v_{r})|}{d_{R}(g_{r}, K \cdot G_{v_{r}})} \leq c'$.  Since this inequality is true for all $s > 0$ we see from (7.1) that $\lambda_{Y}(s) = lim~_{s \rightarrow \infty}~\frac{log|exp(sY)(v)|}{s} \leq c'$.  We conclude that $\lambda^{-}(v) \leq \lambda_{Y}(v) \leq c'$, which contradicts the fact that $\lambda^{-}(v) \geq c > c'$ by the definition of c.
\end{proof}

\begin{proposition}  Let O be the nonempty Zariski open subset of $\br^{n}$ on which $dim~G_{v}$ takes its minimum value.  Let $\fM'$ denote the set of minimal vectors in O.  Let C be a compact subset of $\fM'$.

	1)  The set G(C) is closed in V.
	
	2)  For $A > 0$ define $B_{A} = \{v \in \br^{n} : |v| \leq A \}$ and $X_{A} = B_{A} \cap G(C)$.  Then $X_{A}$ is compact, and there exists a compact set $Y_{A} \subset G$ such that $X_{A} \subset Y_{A}(C)$.
\end{proposition}

$\mathit{Remarks}$

	1) Both parts of this result fail if C is a single point $\{v \}$, where the orbit G(v) is not closed in V.   Part 1) clearly fails in this case, so we address part 2).  Let w be a vector in $\overline{G(v)} - G(v)$ and let $\{g_{r} \} \subset G$ be a sequence such that $g_{r}(v) \rightarrow w$ as $r \rightarrow \infty$.  If $A > |w|$, then $g_{r}(v) \in B_{A}$ for large r, but since w $\in \overline{G(v)} - G(v)$  it is easy to see that we can't write $g_{r}(v) = g'_{r}(v)$ for some sequence $\{g'_{r} \}$ in a compact subset of G. To prove a result of the type of (9.6) we are thus forced to consider only vectors whose G-orbits are closed in $\br^{n}$.  All closed G-orbits must intersect $\fM$ and considering only vectors in $\fM'$ seems to be a reasonable normalizing hypothesis.

	2)  In general, part 1) of (9.6) is false for the set $G(\fM')$.  For example, let $G = GL(n,\br)$ act by conjugation on $M(n,\br)$.  Then $G(\fM')$ contains a nonempty Zariski open subset of $M(n,\br)$, but $G(\fM') \neq M(n,\br)$ since G has nonclosed orbits in $M(n,\br)$.  Hence $G(\fM')$ cannot be closed in $M(n,\br)$.
\newline

	We now begin the proof of (9.6).

	1)  Let $\{g_{r} \} \subset G$ and $\{v_{r} \} \subset C$ be sequences such that $g_{r}(v_{r}) \rightarrow w \in \overline{G(C)}$ as $r \rightarrow \infty$.  If $d_{R}(g_{r}, G_{v_{r}}) \rightarrow \infty$, passing to a subsequence, then $|g_{r}(v_{r})| \rightarrow \infty$ as $r \rightarrow \infty$ by (9.4).  Therefore since $\{|g_{r}(v_{r})| \}$ is bounded there exists B $>0$ such that $d_{R}(g_{r}, G_{v_{r}}) \leq B$ for all r.
	
	Choose $h_{r} \in G_{v_{r}}$ such that $d_{R}(g_{r}, h_{r}) = d_{R}(g_{r}, G_{v_{r}}) \leq B$ for all r.  If $\zeta_{r} = g_{r} h_{r}^{-1}$, then $d_{R}(\zeta_{r}, I) = d_{R}(g_{r},  h_{r}) \leq B$.  By the completeness of $(G,d_{R})$ and the Hopf-Rinow theorem the set $C = \{g \in G : d_{R}(g, I) \leq B \}$ is compact in G.  Let $\zeta_{r} \rightarrow \zeta \in C$ as $r \rightarrow \infty$, passing to a subsequence.   Finally, $w = lim~_{r \rightarrow \infty}~g_{r}(v_{r}) =  lim~_{r \rightarrow \infty}~(\zeta_{r} h_{r})(v_{r}) = \zeta(v) \in G(v) \subset G(C)$.
	
\hspace{.15in}	2) By 1) $X_{A}$ is a closed subset of $B_{A}$, and hence $X_{A}$ is compact.  By (9.3) there exists $c_{1} > 0$ such that $\lambda^{-}(v) \geq c_{1}$ for all v $\in$ C.  Choose $R_{1} > 0$ such that $exp(\frac{R_{1} c_{1}}{2}) > A$.  By making $R_{1}$ still larger we may assume by (9.4) that if v $\in$ C, g $\in$ G are elements such that $d_{R}(g, G_{v}) > R_{1}$, then $\frac{log|g(v)|}{d_{R}(g, G_{v})} \geq \frac{c_{1}}{2}$.   Let $Y_{A} = \{g \in G : d_{R}(g, I) \leq R_{1} \}$.  The set $Y_{A}$ is compact in G by the Hopf-Rinow theorem.

	We assert that $X_{A} \subset Y_{A}(C)$.  Let w $\in X_{A}$ be given.  Then there exists v $\in$ C and g $\in$ G such that $w = g(v)$, and moreover $|g(v)| \leq A$ by the definition of $X_{A}$.  We show first that $d_{R}(g, G_{v}) \leq R_{1}$.  If this were not the case, then by the choice of $R_{1}$ we would have $|g(v)| \geq exp(\frac{c_{1}}{2} d_{R}(g, G_{v})) > exp(\frac{c_{1} R_{1}}{2}) > A$.  This contradiction shows that $d_{R}(g, G_{v}) \leq R_{1}$.
	
	The remainder of the proof is similar to the proof of 1), and we omit some details.  Choose h $\in G_{v}$ such that $d_{R}(g, h) = d_{R}(g, G_{v}) \leq R_{1}$.  If $\zeta = gh^{-1}$, then $d_{R}(\zeta, I) \leq R_{1}$ and $\zeta \in Y_{A}$.  Finally, $w = g(v) = (\zeta h)(v) = \zeta (v) \in Y_{A}(C)$.

\section{Applications}

$\mathit{Criteria~for~detecting~closed~orbits}$	

\begin{proposition}  Let G be a closed, connected, self adjoint, noncompact subgroup of $GL(n,\br)$, and let v be a nonzero vector such that $\lambda^{-}(v) > 0$.  Then G(v) is a closed subset of $\br^{n}$.
\end{proposition}

\begin{proof}  The proof here is also similar to the proof of 1) and 2) in (9.6) and we omit some details. Let $\{g_{r} \}$ be a sequence in G such that $g_{r}(v) \rightarrow w \in \overline{G(v)}$ as $r \rightarrow \infty$.  If $d_{R}(g_{r}, G_{v}) \rightarrow \infty$, passing to a subsequence, then $|g_{r}(v)| \rightarrow \infty$ by 2) and 3) of (8.1) since lim inf   $_{r \rightarrow \infty} \frac{log~|g_{r}(v)|}{d_{R}(g_{r}, G_{v})} \geq \lambda^{-}(v) > 0$.  Since the sequence $\{g_{r}(v) \}$ is bounded in V there exists a positive constant $c_{1}$  such that  $d_{R}(g_{r}, G_{v}) \leq c_{1}$ for r $\geq$ N.

	Let $h_{r} \in G_{v}$ be an element such that $d_{R}(g_{r}, h_{r}) \leq c_{1}$ for all r.  If $\zeta_{r} = g_{r} h_{r}^{-1}$, then $d_{R}(\zeta_{r}, I) \leq c_{1}$.  Let $\zeta_{r} \rightarrow \zeta \in G$, passing to a subsequence if necessary.  Finally, $g_{r}(v) = (\zeta_{r} h_{r})(v) = \zeta_{r}(v) \rightarrow \zeta (v) \in$ G(v).
\end{proof}

\begin{proposition}  Let G be a closed, self adjoint, noncompact subgroup of $GL(n,\br)$ with finitely many connected components, and let $G_{0}$ denote the connected component that contains the identity.   Let v be a nonzero vector of $\br^{n}$ such that the orbit G(v) is unbounded.  Then the following assertions are equivalent.

 	1)  G(v) is closed in $\br^{n}$.
	
	2)  $\lambda^{-}(w) > 0$ for some w $\in$ $G_{0}$(v), where $\lambda^{-} : \br^{n} \rightarrow \br$ is defined by the connected group $G_{0}$.
	
	3)  G$_{0}$(v) is closed in $\br^{n}$.
	
	4)  G(v) contains an element w that is minimal with respect to both G$_{0}$ and G.
\end{proposition}

\begin{proof}  We prove the result in the cyclic order 1) $\Rightarrow$ 4), 4) $\Rightarrow$ 3), 3) $\Rightarrow$ 2) and 2) $\Rightarrow$ 1).

	1) $\Rightarrow$ 4).  If 1) holds, then G(v) contains a minimal vector w by (6.2).  The vector w is also minimal for $G_{0}$ since $G_{0}(w) \subset G(w) = G(v)$.
	
	4) $\Rightarrow$ 3).  Let w $\in$ G(v) be an element that is minimal for both $G_{0}$ and G.  Then $\lambda^{-}(w) > 0$ by 2) of (8.1) applied to $G_{0}$, and $G_{0}(w)$ is closed in $\br^{n}$ by (10.1).  If $w = g(v)$ for g $\in$ G, then $G_{0}(v) = (g^{-1} G_{0} g)(v) = g^{-1} G_{0}(w)$ is closed in $\br^{n}$.
	
	3)  $\Rightarrow$ 2).  If $G_{0}(v)$ is closed in $\br^{n}$, then $G_{0}$ contains a minimal vector w by (6.2), and $\lambda^{-}(w) > 0$ by 2) of (8.1).
	
	2)  $\Rightarrow$ 1).  If $\lambda^{-}(w) > 0$ for some w $\in G_{0}(v)$, then $G_{0}(v) = G_{0}(w)$ is closed in $\br^{n}$ by (10.1).  Since G has finitely many connected components we may write $G = \bigcup_{i=1}^{n} g_{i} G_{0}$ for suitable elements $\{g_{1}, ... , g_{N} \}$.  It follows that $G(v) = \bigcup_{i=1}^{N} g_{i}~G_{0}(v)$ is closed in $\br^{n}$.
\end{proof}

$\mathit{An~ application~ to~ representation~theory}$
\newline

	A more general version of the next result is known (see for example (2.1) in [RS]), but the statement is more complicated and the proof is less elementary.

\begin{theorem} Let V be a finite dimensional vector space over $\br$. Let G be a connected, noncompact, semisimple Lie group, and let $\rho : G \rightarrow GL(V)$ be a C$^{\infty}$ homomorphism.  Then there exists an inner product $\langle , \rangle$ on V such that $\rho(G)$ is invariant under the involutive automorphism $\theta : GL(V) \rightarrow$ GL(V) given by $\theta(g) = (g^{t})^{-1}$.
\end{theorem}

$\mathit{Remark}$  As usual, $g^{t} : V \rightarrow V$ denotes the metric transpose of $g : V \rightarrow V$ relative to the inner product $\langle , \rangle$.  It is well known that homomorphic images of semisimple groups are also semisimple, and hence the group $H = \rho(G)$ is a connected, semisimple subgroup of GL(V).  If we fix an orthonormal basis $\{v_{1},~...~,v_{n} \}$ of V relative to $\langle , \rangle$, then we obtain a C$^{\infty}$ homomorphism $\rho' : G \rightarrow GL(n,\br)$ given  by $\rho'(g)_{ij} = \langle \rho(g)(v_{j}),v_{i} \rangle$ for $1 \leq i , j \leq n$.  The subgroup $H' = \rho'(G)$ of $GL(n,\br)$ is self adjoint and semisimple, and the semisimplicity implies that H$'$ is a closed subgroup of $GL(n,\br)$ by the main theorem in section 6 of [M1].  Hence we may apply the results above to the subgroup $H' = \rho'(G)$ of $GL(n,\br)$

\begin{proof}   The semisimplicity of $\fH = d\rho(\fG)$ implies that $\fH = [\fH,\fH]$, the derived algebra of $\fH$.  By Theorem 15, section 14 of [C] the group $H = \rho(G)$ is algebraic (see also Corollary 7.9, chapter II of [B]).  The assertion of the theorem now follows from the main result of [M2].
\end{proof}

\section{Appendix I}

$\mathit{The~asymmetry~of~d_{r}~and~d_{L}~in~the~main~result}$
\newline
	
	We first obtain necessary conditions for the main result to hold if $d_{R}(g, G_{v})$ is replaced by $d_{L}(g, G_{v})$.  More precisely we show
	
\vspace{.2in}
	
\begin{lemma}
\end{lemma}  Let v be a nonzero vector in $\br^{n}$ such that $\lambda^{-}(v) > 0$.  Suppose there exist positive constants a,b so that 3) of (8.1) holds if a replaces $\lambda^{-}(v)$, b replaces $\lambda^{+}(v)$ and $d_{L}(g, G_{v})$ replaces $d_{R}(g, G_{v})$.  Then there exist positive constants A,C such that if $d_{R}(g, G_{v}) \geq A$ and  $d_{L}(g, G_{v}) \geq A$, then

	(a)  $\frac{1}{C} \leq \frac{d_{R}(g, G_{v})}{d_{L}(g, G_{v})} \leq C$
	
	(b)  $\frac{1}{C} \leq \frac{log|g^{t}(v)|}{log|g(v)|} \leq C$
\newline
	
$\mathit{Remark}$  The proof of the lemma will show that the conditions (a) and (b) of the lemma are also sufficient for the replacement of $d_{R}$ by $d_{L}$ to hold in the statement of (8.1).
	
\begin{proof}

	By 3) of (8.1) and the hypothesis of the lemma there exist positive constants $a_{1},b_{1},a_{2},b_{2}$ and A such that
	
	(1) \hspace{.5in}  $a_{1} \leq \frac{log~|g(v)|}{d_{R}(g, G_{v})} \leq b_{1}$ \hspace{.3in}  if $d_{R}(g, G_{v}) \geq A$
	
	 (2) \hspace{.5in}  $a_{2} \leq \frac{log~|g(v)|}{d_{L}(g, G_{v})} \leq b_{2}$ \hspace{.3in}  if $d_{L}(g, G_{v}) \geq A$

	From (1) and (2) we obtain
	
	(3) \hspace{.33in}  $\frac{a_{2}}{b_{1}} \leq \frac{d_{R}(g, G_{v})}{d_{L}(g, G_{v})} \leq \frac{b_{2}}{a_{1}}$ \hspace{.3in}  if $d_{R}(g, G_{v}) \geq A$ and  $d_{L}(g, G_{v}) \geq A$
	
From (3) we obtain assertion (a).
\newline

	To prove assertion (b) we recall from (4.1) that $d_{R}(g,h) = d_{L}(g^{t},h^{t})$ for all g,h $\in$ G and recall from (1) of (6.1) that $G_{v} = G_{v}^{t}$.  If $d_{R}(g, G_{v}) \geq A$ and  $d_{L}(g, G_{v}) \geq A$, then $d_{R}(g^{t}, G_{v})=  d_{L}(g, G_{v}^{t}) = d_{L}(g, G_{v}) \geq A $ and similarly $d_{L}(g^{t}, G_{v})=  d_{R}(g, G_{v}) \geq A$.  From (1) and (2) we obtain
	
	(4)  \hspace{.33in}  $a_{1} \leq \frac{log~|g^{t}(v)|}{d_{R}(g^{t}, G_{v})} = \frac{log~|g^{t}(v)|}{d_{L}(g, G_{v})} \leq b_{1}$ \hspace{.2in} if $d_{R}(g, G_{v}) \geq A$ and  $d_{L}(g, G_{v}) \geq A$
\newline

	From (2) and (4) we obtain
	
	$\frac{a_{1}}{b_{2}} \leq \frac{log|g^{t}(v)|}{log|g(v)|} \leq \frac{b_{1}}{a_{2}}$ if $d_{L}(g, G_{v}) \geq A$ and  $d_{R}(g, G_{v}) \geq A$.
	
This proves assertion (b) and completes the proof of  (11.1).
\end{proof}
	
	Next we  show that the second condition of (11.1) fails for a certain vector v $\in M(3,\br)
\approx \br^{9}$ if $G = GL(3,\br)$ acts on $M(3,\br)$ by conjugation.  A similar argument shows that the second condition of (11.1) fails for any n $\geq 3$ for the action of $GL(n,\br)$ on $M(n, \br)$ by conjugation.

	Let $v = \left (\begin{array} {ccc} 0 & 0 & 0 \\ 0 & 0 & 1\\ 0 & -1 & 0 \\ \end{array} \right)$, and let  $k = \left (\begin{array} {ccc} 0 & 1 & 0 \\ 1 & 0 & 0\\ 0 & 0 & -1 \\ \end{array} \right) \in SO(3,\br)$.  \newline For each positive integer N let $B_{N} = diag(N+2, 2 ,1) = \left (\begin{array} {ccc} N+2 & 0 & 0 \\ 0 & 2 & 0\\ 0 & 0 & 1 \\ \end{array} \right)$.  Let $g_{N}(s) = k~ exp(sB_{N})$.  We will show
	
	(i)  $\lambda^{-}(v) > 0$.
	
	(ii) $\frac{log |g_{N}(s)^{t}(v)|}{log |g_{N}(s)(v)|} \rightarrow N+1$ as s $\rightarrow \infty$.
	
	(iii)  $d_{R}(g_{N}(s), G_{v}) \rightarrow \infty$ and $d_{R}(g_{N}(s)^{t}, G_{v})  \rightarrow \infty$ as s $\rightarrow \infty$
	
	 Since N is arbitrary we will obtain a contradiction to the uniform boundedness condition (b) of (11.1).
\newline
	
	 We prove (i).  The Lie algebra $\fG = M(3,\br)$ acts on $M(3,\br)$ by the adjoint action ; i.e. if A,B $\in M(3,\br)$, then $A(B) = A B - B A$.   By example 3 in section 1 of [EJ] a vector Z $\in M(3,\br)$ is minimal for the G action above $\Leftrightarrow Z Z^{t} = Z^{t} Z$.  In particular the skew symmetric matrix v above is minimal, and $\lambda^{-}(v) > 0$ by (7.6).
	
	 We prove (ii).  If $A = diag(\lambda_{1}, \lambda_{2}, \lambda_{3})$, then $A(E_{ij}) = A E_{ij} - E_{ij} A = (\lambda_{i} - \lambda_{j}) E_{ij}$ for all i $\neq$ j , where $E_{ij}$ is the matrix with 1 in position ij and zeros elsewhere.  Observe that the only nonzero components of v lie in the 23 and 32 positions.   Hence by inspection and the definition of $\lambda_{B_{N}}(v)$  we obtain
	
	(a)  $\lambda_{B_{N}}(v) = (\lambda_{2} - \lambda_{3}) = 1$.

    The elements k of $SO(3,\br)$ preserve the lengths of vectors in $M(3,\br)$, and hence we have

    (b)  $|g_{N}(s)(v)| = |exp(sB_{N})(v)|$ for all s
	
	 If $B_{N}'  = diag(2, N+2, 1)= k~B_{N}~k^{-1}$, then we have
	
	 (c) $\lambda_{B_{N}'}(v) = (\lambda_{2} - \lambda_{3}) = N+1$

	(d) $|g_{N}(s)^{t}(v)| = |exp(sB_{N})~k^{-1}(v)| = |k^{-1}\{k~exp(sB_{N})~k^{-1}) \} (v) = |exp(sB_{N}')(v)|$ for all s.
	
	Finally $\frac{log |g_{N}(s)^{t}(v)|}{log |g_{N}(s)(v)|} = \frac{log |exp(sB_{N}')(v)|}{s} \frac{s}{log |exp(sB_{N}(v)|} \rightarrow \frac{\lambda_{B_{N}'}(v)}{\lambda_{B_{N}}(v)} = N+1$ as s $\rightarrow \infty$ by (a), (b), (c), (d) and (7.1).  This completes the proof of (ii).
\newline
	
	We prove (iii). We shall need a preliminary result.

\begin{lemma}  Let v be a nonzero minimal vector in $\br^{n}$, and let X $\in \fP$ be an element such that $X(v) \neq 0$.  Then

	a)  $d_{R}(exp(sX), G_{v}) \rightarrow \infty$ as s $\rightarrow \infty$
	
	b)  $|\exp(sX)(v)| \rightarrow \infty$ as s $\rightarrow \infty$.
\end{lemma}

\begin{proof}  Proof of a).  We suppose that the assertion of a) is false for some nonzero minimal vector v and some X $\in \fP$ with X(v) $\neq 0$.  Then there exist sequences $g_{k} \subset G$, $\zeta_{k} \subset G_{v}$ and $\{s_{k} \} \subset \br$  and a positive number A such that $s_{k} \rightarrow \infty$ as k $\rightarrow \infty$ and $d_{R}(exp(s_{k} X), \zeta_{k}) \leq A$ for all k.  If $\varphi_{k} = exp(s_{k} X) \zeta_{k}^{-1}$, then $d_{R}(\varphi_{k}, Id) = d_{R}(exp(s_{k} X), \zeta_{k}) \leq A$ for all k.  Let $C = \{\varphi \in G : d_{R}(\varphi, Id) \leq A \}$.  Then C is compact by the completeness of $(G, d_{R})$, and there exists a positive number c such that $|\varphi(v)| \leq c$ for all $\varphi \in C$.  Since $\{\varphi_{k} \} \subset C$ we obtain $|exp(s_{k} X) (v)| = |\varphi_{k} \zeta_{k} (v)| = |\varphi_{k} (v)| \leq c$ for all k.

	The function $f_{X}(s) = |exp(sX)(v)|^{2}$  satisfies  $f_{X}'(0) =  2 \langle X(v) , v \rangle = 0$  since v is minimal.  Moreover, $f_{X}''(s) = 4 |X exp(sX)(v)|^{2} \geq 0$, and in particular $f_{X}''(0) = 4 |X(v)| > 0$ since $X(v) \neq 0$ by hypothesis.  It follows from the convexity of $f_{X}$ that $f_{X}(s) \rightarrow \infty$ as s $\rightarrow \infty$, which contradicts the fact that $f_{X}(s_{k}) \leq c^{2}$ for all k.  This completes the proof of a).
	
	Proof of b).  This is contained in the last paragraph of the proof of a).
\end{proof}

  We now complete the proof of (iii).  Recall that $B_{N} = diag (N+2, 2,1)$ and $B_{N}' = diag (2, N+2 , 1).$ Hence $B_{N}(v) = B_{N}(E_{23}) - B_{N}(E_{32}) = (\lambda_{2} - \lambda_{3}) E_{23} - (\lambda_{3} - \lambda_{2})(E_{32}) = E_{23} + E_{32} \neq 0$.  Similarly, $B_{n}'(v) = (N+1) (E_{23} + E_{32}) \neq 0$.

  	Using 4) of (4.5), the triangle inequality and the lemma above we obtain $d_{R}(g_{N}(s), G_{v}) = d_{R}(k~exp(sB_{N}), G_{v}) = d_{R}(exp(sB_{N}), k^{-1} \cdot G_{v}) \geq d_{R}(exp(sB_{N}), G_{v}) - d_{R}(k^{-1}, Id) \rightarrow \infty$ as s $\rightarrow \infty$.  Similarly $d_{R}(g_{N}(s)^{t}, G_{v}) = d_{R}(exp(sB_{N})~k^{-1}, G_{v}) = d_{R}(k^{-1}~exp(sB_{n}'), G_{v}) = d_{R}(exp(sB_{N}'), k \cdot G_{v}) \geq d_{R}(exp(sB_{N}') , G_{v}) - d_{R}(k , Id) \rightarrow \infty$ as s $\rightarrow \infty$.  This completes the proof of (iii).

\section{Appendix II}
				
	We give here the proofs of Propositions 5.1 and 6.5.
	
$\mathit{Proof~of~Proposition~5.1}$

	Assertion 1) of Proposition 5.1 is obvious and the remaining assertions are an immediate consequence of the next two results.
	
 $\mathbf{Lemma~5.1A}$  $\fZ(\fG)$ and $\fG_{0}$ are self adjoint ideals, and $\fG_{0}$ has trivial center.

  $\mathbf{Lemma~5.1B}$  Let $\fG$ be a self adjoint subalgebra of $M(n,\br)$ that has trivial center.  Then $\fG$ is semisimple.
 \newline

	We prove Lemma 5.1A.  Let $X \in \fG$ and $Z \in \fZ(\fG)$ be given.  Then $[Z^{t},X] = [X^{t},Z]^{t} = 0$ since $X^{t} \in \fG$.  Hence $\fZ(\fG)$ is self adjoint.  Let $Z \in \fZ(\fG)$ and $X \in \fG_{0}$ be given.  Then $\langle Z , X^{t} \rangle  =  \langle Z^{t} , X \rangle = 0$ since $Z^{t} \in \fZ(\fG)$.  Hence $\fG_{0}$ is self adjoint.
	
	Clearly $\fZ(\fG)$ is an ideal.  Let $X \in \fG,Y \in \fG_{0}$ and $Z \in \fZ(\fG)$ be given.  Then using 3) of (4.2) we obtain $\langle [X,Y] , Z \rangle = \langle ad~X(Y) , Z \rangle = \langle Y , (ad~X)^{*}(Z) \rangle = \langle Y , (ad~X^{t})(Z) \rangle = 0$ since $X^{t} \in \fG$.  Hence $\fG_{0}$ is an ideal of $\fG$.  It follows from 1) of (5.1) that $\fG_{0}$ has trivial center, which completes the proof of Lemma 5.1A.
	
	We prove Lemma 5.1B.  Since $\fG$ is self adjoint we may write $\fG = \fK \oplus \fP$, where $\fK = \{X \in \fG : X^{t} = - X \}$ and $\fP = \{X \in \fG : X^{t} = X \}$.    Let B denote the Killing form of $\fG$.  It suffices to prove a)  $B(\fK , \fP) = \{0 \}$  b)  B is negative definite on $\fK$ and c)  B is positive definite on $\fP$.
	
	a)  Let $X \in \fK$ and $Y \in \fP$ be given.  Then $(ad~ X \circ ad~ Y)(\fP) \subset \fK$ and $(ad~ X \circ ad~ Y)(\fK) \subset \fP$ since $[\fK , \fK] \subset \fK, [\fK , \fP] \subset \fP$ and $[\fP , \fP] \subset \fK$.  Compute the matrix of ad X $\circ$ ad Y relative to a basis of $\fG$ that is a union of bases of $\fK$ and $\fP$.  The diagonal elements of the matrix are all zero, and it follows that B(X,Y) $=$ trace ad X $\circ$ ad Y $= 0$.
	
	b)  If $X \in \fK$, then by 3) of (4.2) we obtain B(X,X) $=$ trace ad X $\circ$ ad X $ = -$ trace $ad X \circ (ad X)^{*} \leq 0$ with equality $\Leftrightarrow$ ad X $\equiv 0 \Leftrightarrow X = 0$ since $\fG$ has trivial center.  This proves b) and the proof of c) is similar.  This completes the proof of Lemma 5.1B.
\newline

$\mathit{Proof~of~Proposition~6.5}$

	We prove the result in two cycles : $1) \Rightarrow 4)~\Rightarrow 3)~\Rightarrow 2)~\Rightarrow 1)$ and $4) \Rightarrow 5) \Rightarrow 4)$.
	
	$1) \Rightarrow 4)$  Since 1) holds there exists a positive constant c such that $|g(v)| \leq c$ for all g $\in$ G.  Let X be any nonzero element of $\fP$ and define $f_{X}(t) = |exp(tX)(v)|^{2}$.  Then $f_{X}''(t) = 4 |Xexp(tX)(v)|^{2} \geq 0$ for all t $\in \br$.  By hypothesis $f_{X}(t) \leq c$ for all t $\in \br$, and hence $f_{X}(t) \equiv$ constant by the convexity of $f_{X}$.  It follows that $0 = f_{X}''(0) = 4 |X(v)|^{2}$.

	$4) \Rightarrow 3)$  Let g $\in$ G be given.  By the KP decomposition there exist elements k $\in$ K and X $\in \fP$ such that $g = k exp(X)$.  The hypothesis 4) implies that $exp(X) \in G_{v}$ for all X $\in \fP$, and it follows immediately that $G = K \cdot G_{v}$.
	
	The assertion $3) \Rightarrow 2)$ is obvious.  We prove that $2) \Rightarrow 1)$.  Let A be a positive constant such that $d_{R}(g, K \cdot G_{v}) \leq A$ for all g $\in$ G.  It suffices to show that $|g(v)| \leq exp(A) |v|$ for all g $\in$ G.
	
	Let g $\in$ G be given.  By (5.6) there exist elements k $\in$ K, h $\in G_{v}$ and X $\in \fP_{v}^{\perp}$ such that $g = k exp(X) h$ and $|X| = d_{R}(g, K \cdot G_{v}) \leq A$.  Hence $|\lambda| \leq |X| \leq A$ if $\lambda$ is any eigenvalue of X.  Write $v = \sum_{i=1}^{N} v_{i}$, where $X(v_{i}) = \lambda_{i} v_{i}$ for some eigenvalues $\lambda_{i}, 1 \leq i \leq N$.  Then $exp(X)(v) = \sum_{i=1}^{N} exp(\lambda_{i}) v_{i}$ and $|g(v)|^{2} = |exp(X)(v)|^{2} = \sum_{i=1}^{N} exp(2 \lambda_{i}) |v_{i}|^{2} \leq \sum_{i=1}^{N} exp(2A) |v_{i}|^{2} = exp(2A) |v|^{2}$.  It follows that $|g(v)| \leq exp(A) |v|$ for all g $\in$ G.  Hence 2) $\Rightarrow$ 1).
\newline

    We next show that $4) \Rightarrow 5) \Rightarrow 4)$.  If 4) holds then $\fP \subset \fG_{v}$ and $\fG = \fK \oplus \fP \subset \fK + \fG_{v} \subset \fG$, and equality must hold everywhere.  Hence 4) $\Rightarrow 5)$.

    Now suppose that 5) holds.  It suffices to show that v is a minimal vector.  Then $\fG_{v} = \fK_{v} \oplus \fP_{v}$ by 2) of (6.1), and it follows that $\fG = \fK + \fG_{v} = \fK \oplus \fP_{v} \subset \fK \oplus \fP = \fG$.  Equality must hold everywhere, and this implies that $\fP = \fP_{v}$, which is 4).

    We show that v is minimal if 5) holds.  Let X $\in \fG$ be given, and write $X = K + H$, where $K \in \fK$ and $H \in \fG_{v}$.  Then $\langle m(v) , X \rangle = \langle X(v) , v \rangle =\langle K(v) , v \rangle = 0$ since K is skew symmetric and $H(v) = 0$.  Hence $m(v) = 0$ since $X \in \fG$ was arbitrary, and v is minimal by 1) of (6.2).
\newline
	
	To complete the proof of (6.5) it remains only to prove that if G has no nontrivial compact, normal subgroups, then G fixes v if any of the conditions above is satisfied.  Let these conditions be satisfied.  Then X(v) $= 0$ for all X $\in \fP$, and it follows immediately that X(v) $= 0$ for all X $\in [\fP, \fP] \subset \fK$.  It suffices to prove that $[\fP , \fP] = \fK$, for then X(v) $= 0$ for all X $\in \fG = \fK \oplus \fP$.  This is an immediate consequence of the next result.
	
$\mathbf{Lemma}$  Let $\fK_{1} = \{X \in \fK : \langle X , [Y,Z] \rangle = 0~\rm{for~all}~Y,Z \in \fP \}$.  Let $K_{1}$ be the connected Lie subgroup of K whose Lie algebra is $\fK_{1}$.  Then $\overline{K_{1}}$, the closure of $K_{1}$ in G, is a compact normal subgroup of G.

\begin{proof}  It suffices to prove that $\fK_{1}$ is an ideal of $\fG = \fK \oplus \fP$, or equivalently, that $ad~\xi(\fK_{1}) \subset \fK_{1}$ for $\xi \in \fK~\cup~\fP$.  Let $\xi \in \fK, X \in \fK_{1}$ and $Y,Z \in \fP$.  Then $ad~\xi(X) \in [\fK , \fK] \subset \fK$.  By 3) of (4.2) we obtain $\langle ad~\xi(X) , [Y,Z] \rangle = - \langle X , ad~\xi([Y,Z]) \rangle = - \langle X , [ad~\xi(Y),Z] \rangle - \langle X , [Y, ad~\xi(Z)] \rangle = 0$ since $X \in \fK_{1}$ and $ad~\xi(\fP) \subset [\fK, \fP] \subset \fP$.  Hence $ad~\xi(\fK_{1}) \subset \fK_{1}$ if $\xi \in \fK$.

	Let $\xi , Y \in \fP$ and $X \in \fK_{1}$.  Then $\langle ad~\xi(X) , Y \rangle = \langle X , ad~\xi(Y) \rangle = 0$ by 3) of (4.2) and the definition of $\fK_{1}$.  This proves that $ad~\xi(\fK_{1}) = \{0 \}$ if $\xi \in \fP$ since $ad~\xi(X) \in [\fP , \fK] \subseteq \fP$ and $Y \in \fP$ was arbitrary.
\end{proof}

$\mathit{References}$

[B]  A. Borel, $\mathit{Linear~algebraic~groups}$, Springer, New York, 1991.
\newline

[B-HC]  A. Borel and Harish-Chandra, " Arithmetic subgroups of algebraic groups ", Annals of Math (2) 75 (1962), 485-535.
\newline

[C]  C. Chevalley , $\mathit{Theorie~des~Groupes~de~Lie~:~Tome~II,~Groupes~Algebriques}$, Universit$\acute{e}$ de Nancago, Hermann, Paris, 1951.
\newline

[CE]  J. Cheeger and D. Ebin, $\mathit{Comparison~Theorems~in~Riemannian~Geometry}$, North Holland, Amsterdam, 1975
\newline

[E]  P. Eberlein, " Riemannian 2-step Nilmanifolds with Prescribed Ricci tensor ", Contemp. Math., vol. 469, (2008), 167-195.
\newline

[EJ]  P. Eberlein and M. Jablonski, " Closed Orbits of Semisimple Group Actions and the Real Hilbert-Mumford Function ", Contemp. Math., vol.491, (2009), 283-322.
\newline

[KN]  G. Kempf and L. Ness, " The length of vectors in representation spaces ", Lecture Notes in Mathematics 273, 233-243, Springer, New York, 1979.
\newline

[M]  A. Marian, " On the real moment map ", Math. Res. Lett. 8, (2001), 779-788.
\newline

[M1]  G.D. Mostow, " The extensibility of local Lie groups of transformations and groups on surfaces ", Annals of Math., 52 (1950), 606-636.
\newline

[M2] -------, " Self adjoint groups ", Annals of Math., 62 (1955), 44-55.
\newline

[RS]  R. Richardson and P. Slodowy, " Minimum vectors for real reductive algebraic groups ", Jour. London Math. Soc. 42 (1990), 409-429.
\end {document}